\documentclass[11pt,a4paper]{article}

\usepackage{geometry}
\usepackage[T1]{fontenc}
\usepackage[utf8]{inputenc}
\usepackage{authblk}

\usepackage{graphics}
\usepackage{tikz,pgf}
\usetikzlibrary{decorations.markings, decorations.pathmorphing}
\usetikzlibrary{arrows,decorations,backgrounds,scopes,plotmarks}
\usepackage{subfigure}

\usepackage[english]{babel}
\usepackage[all]{xy}
\usepackage{marvosym}
\usepackage{stmaryrd,mathrsfs,amsmath,amssymb,bm,amsfonts}
\usepackage{enumerate}
\usepackage{amsmath,amsthm,amssymb,amscd}

\newtheorem{thm}{Theorem}[section]
\newtheorem{defn}[thm]{Definition}
\newtheorem{prop}[thm]{Proposition}
\newtheorem{cor}[thm]{Corollary}
\newtheorem{lem}[thm]{Lemma}
\newtheorem{rem}[thm]{Remark}

\def \CPP{CPP-}
\def \UPO{UPO-}

\numberwithin{equation}{subsection}

\begin{document}

\title{Combinatorial characterization of upward planarity}

\author[a,b]{Xuexing Lu}

\author[a,b]{Yu Ye}

\affil[a]{\small School of Mathematical Sciences, University of Science and Technology of China, Hefei, China}
\affil[b]{Wu Wen-Tsun Key Laboratory of Mathematics, Chinese Academy of Sciences, Hefei, China}

\renewcommand\Authands{ and }
\maketitle

\begin{abstract}

We give a combinatorial characterization of upward planar graphs in terms of upward planar orders, which are special linear extensions of edge posets.

\end{abstract}

\text{\emph{Keywords}: upward planar graph,  planar $st$ graph, upward planar order}

\text{\emph{Mathematics Subject Classification (2010)}: 	05C10, 06A99 }

\tableofcontents
\section{Introduction}
A planar drawing of a directed graph is \emph{upward} if all edges increase monotonically in the vertical direction (or other fixed direction).
A directed graph is called \emph{upward planar} if it admits an upward planar drawing; see Fig. $1$ for an example. Clearly, an upward planar graph is necessarily acyclic.  A directed graph together with an upward planar drawing is called an \emph{upward plane graph}. They are commonly used to represent hierarchical structures, such as PERT networks, Hasse diagrams, family trees, etc, and have been extensively studied in the fields of graph theory, graph drawing algorithms, and ordered set theory (see, e.g., \cite{[GT95]} for a review).

\begin{center}
\begin{tikzpicture}[scale=0.25]

\node (v2) at (0.5,0.5) {};
\node (v1) at (-3,-3) {};
\node (v3) at (3,-1.5) {};
\node (v4) at (7,2) {};
\node (v8) at (0.5,-2) {};
\node (v5) at (-0.5,-7.5) {};
\node (v6) at (3.5,-5.5) {};
\node (v7) at (6,-8.5) {};
\draw  (-3,-3) -- (0.5,0.5)[postaction={decorate, decoration={markings,mark=at position .5 with {\arrow[black]{stealth}}}}];
\draw  (3,-1.5) -- (0.5,0.5)[postaction={decorate, decoration={markings,mark=at position .5 with {\arrow[black]{stealth}}}}];
\draw  (3,-1.5) -- (7,2)[postaction={decorate, decoration={markings,mark=at position .5 with {\arrow[black]{stealth}}}}];
\draw  (-0.5,-7.5) -- (3.5,-5.5)[postaction={decorate, decoration={markings,mark=at position .5 with {\arrow[black]{stealth}}}}];
\draw   (6,-8.5) -- (3.5,-5.5)[postaction={decorate, decoration={markings,mark=at position .5 with {\arrow[black]{stealth}}}}];
\draw  (-0.5,-7.5) -- (-3,-3)[postaction={decorate, decoration={markings,mark=at position .5 with {\arrow[black]{stealth}}}}];
\draw  (-0.5,-7.5) -- (0.5,-2)[postaction={decorate, decoration={markings,mark=at position .5 with {\arrow[black]{stealth}}}}];
\draw  (3.5,-5.5)-- (0.5,-2)[postaction={decorate, decoration={markings,mark=at position .5 with {\arrow[black]{stealth}}}}];
\draw   (6,-8.5) -- (7,2)[postaction={decorate, decoration={markings,mark=at position .5 with {\arrow[black]{stealth}}}}];
\node (v9) at (2,-11) {};
\draw  (2,-11) --  (6,-8.5)[postaction={decorate, decoration={markings,mark=at position .5 with {\arrow[black]{stealth}}}}];
\draw  (2,-11) -- (-0.5,-7.5)[postaction={decorate, decoration={markings,mark=at position .5 with {\arrow[black]{stealth}}}}];
\draw  (3.5,-5.5) -- (7,2)[postaction={decorate, decoration={markings,mark=at position .5 with {\arrow[black]{stealth}}}}];
\draw[fill] (v1) circle [radius=0.2];
\draw[fill] (v2) circle [radius=0.2];
\draw[fill] (v3) circle [radius=0.2];
\draw[fill] (v4) circle [radius=0.2];
\draw[fill] (v5) circle [radius=0.2];
\draw[fill] (v6) circle [radius=0.2];
\draw[fill] (v7) circle [radius=0.2];
\draw[fill] (v8) circle [radius=0.2];
\draw[fill] (v9) circle [radius=0.2];

\end{tikzpicture}

Figure $1$. An upward planar graph.
\end{center}

A first simple characterization of upward planarity was given independently by Di Battista, Tamassia \cite{[BT88]} and Kelly \cite{[Ke87]}.  They characterized upward planar graphs as spanning subgraphs of \emph{planar $st$ graphs}, as shown in Fig. $2$, where a planar $st$ graph is a directed planar graph with exactly one source $s$, exactly one sink $t$ and a distinguished edge $e$ connecting $s$, $t$.
\begin{center}
\begin{tikzpicture}[scale=0.25]
\node (v2) at (0.5,0.5) {};
\node (v1) at (-3,-3) {};
\node (v3) at (3,-1.5) {};
\node (v4) at (7,2) {};
\node (v8) at (0.5,-2) {};
\node (v5) at (-0.5,-7.5) {};
\node (v6) at (3.5,-5.5) {};
\node (v7) at (6,-8.5) {};
\draw  (-3,-3) -- (0.5,0.5)[postaction={decorate, decoration={markings,mark=at position .5 with {\arrow[black]{stealth}}}}];
\draw  (3,-1.5) -- (0.5,0.5)[postaction={decorate, decoration={markings,mark=at position .5 with {\arrow[black]{stealth}}}}];
\draw  (3,-1.5) -- (7,2)[postaction={decorate, decoration={markings,mark=at position .5 with {\arrow[black]{stealth}}}}];
\draw  (-0.5,-7.5) -- (3.5,-5.5)[postaction={decorate, decoration={markings,mark=at position .5 with {\arrow[black]{stealth}}}}];
\draw   (6,-8.5) -- (3.5,-5.5)[postaction={decorate, decoration={markings,mark=at position .5 with {\arrow[black]{stealth}}}}];
\draw  (-0.5,-7.5) -- (-3,-3)[postaction={decorate, decoration={markings,mark=at position .5 with {\arrow[black]{stealth}}}}];
\draw  (-0.5,-7.5) -- (0.5,-2)[postaction={decorate, decoration={markings,mark=at position .5 with {\arrow[black]{stealth}}}}];
\draw  (3.5,-5.5)-- (0.5,-2)[postaction={decorate, decoration={markings,mark=at position .5 with {\arrow[black]{stealth}}}}];
\draw   (6,-8.5) -- (7,2)[postaction={decorate, decoration={markings,mark=at position .5 with {\arrow[black]{stealth}}}}];
\node (v9) at (2,-11) {};
\draw  (2,-11) --  (6,-8.5)[postaction={decorate, decoration={markings,mark=at position .5 with {\arrow[black]{stealth}}}}];
\draw  (2,-11) -- (-0.5,-7.5)[postaction={decorate, decoration={markings,mark=at position .5 with {\arrow[black]{stealth}}}}];
\draw  (3.5,-5.5) -- (7,2)[postaction={decorate, decoration={markings,mark=at position .5 with {\arrow[black]{stealth}}}}];
\draw[fill] (v1) circle [radius=0.2];
\draw[fill] (v2) circle [radius=0.2];
\draw[fill] (v3) circle [radius=0.2];
\draw[fill] (v4) circle [radius=0.2];
\draw[fill] (v5) circle [radius=0.2];
\draw[fill] (v6) circle [radius=0.2];
\draw[fill] (v7) circle [radius=0.2];
\draw[fill] (v8) circle [radius=0.2];
\draw[fill] (v9) circle [radius=0.2];
\draw [dashed] (0.5,-2) --  (0.5,0.5)[postaction={decorate, decoration={markings,mark=at position .5 with {\arrow[black]{stealth}}}}];
\draw [dashed] (3.5,-5.5) -- (3,-1.5)[postaction={decorate, decoration={markings,mark=at position .5 with {\arrow[black]{stealth}}}}];
\draw [dashed](0.5,0.5) -- (7,2)[postaction={decorate, decoration={markings,mark=at position .5 with {\arrow[black]{stealth}}}}];
\node at (7,3) {$t$};
\node at (2,-12) {$s$};
\draw [dashed]  plot[smooth, tension=.7] coordinates {(v9) (6,-10) (9,-6) (9,-3) (8,0) (v4)}[postaction={decorate, decoration={markings,mark=at position .6 with {\arrow[black]{stealth}}}}];
\node at (8.2,-5) {$e$};
\end{tikzpicture}

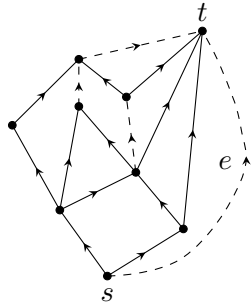
Figure $2$. A planar $st$ graph with the graph in Fig. $1$ as spanning subgraph.
\end{center}

Another fundamental characterization of upward planar graphs was given in \cite{[BB91],[BBLM94]} by means of \emph{bimodal planar drawings} \cite{[GT95]} and \emph{consistent assignments} of sources and sinks to faces. Aware of these elegant and important combinatorial characterizations, in this paper, we will study upward planarity in a different approach.

The notion of a \emph{processive plane graph} (called a \emph{PPG} for short; see Definition \ref{ppg}) was introduced in \cite{[HLY16]} as  a graphical tool for tensor calculus in semi-groupal categories. It turns out that a PPG is essentially equivalent to an upward plane $st$ graph, as shown in Fig. $3$.
\begin{center}
$$
\begin{matrix}
\begin{matrix}
\begin{tikzpicture}[scale=0.3]
\node (v2) at (-4,3) {};
\draw[fill] (-1.5,5.5) circle [radius=0.15];
\node (v1) at (-1.5,5.5) {};
\node (v7) at (-1.5,1) {};
\node (v9) at (1.5,5.5) {};
\node (v14) at (2,1.5) {};
\node (v3) at (-3,7.5) {};
\node (v4) at (-2,7.5) {};
\node (v5) at (-0.5,7.5) {};
\node (v6) at (-4.8,7.5) {};
\node (v11) at (-4.5,-1) {};
\node (v12) at (-2,-1) {};
\node (v13) at (0,-1) {};
\node (v15) at (2,-1) {};
\node (v8) at (1,7.5) {};
\node (v10) at (2.5,7.5) {};
\node  at (-2.5,3.5) {};
\node  at (-3,5.2) {};
\node  at (-1.2,3.3) {};
\node  at (0.5,3.25) {};
\node  at (2.2,3.7) {};
\node  at (-3,1.7) {};
\draw[fill] (-4,3) circle [radius=0.15];
\draw[fill] (v1) circle [radius=0.15];
\draw[fill] (v7) circle [radius=0.15];
\draw[fill] (v9) circle [radius=0.15];
\draw[fill] (v14) circle [radius=0.15];
\draw[fill] (v1) circle [radius=0.15];
\draw[fill] (v2) circle [radius=0.15];
\draw[fill] (v3) circle [radius=0.15];
\draw[fill] (v4) circle [radius=0.15];
\draw[fill] (v5) circle [radius=0.15];
\draw[fill] (v6) circle [radius=0.15];
\draw[fill] (v8) circle [radius=0.15];
\draw[fill] (v10) circle [radius=0.15];
\draw[fill] (v11) circle [radius=0.15];
\draw[fill] (v12) circle [radius=0.15];
\draw[fill] (v13) circle [radius=0.15];
\draw[fill] (v15) circle [radius=0.15];

\draw  plot[smooth, tension=1] coordinates {(v1) (-2.5,5)  (-3.5,4) (v2)}[postaction={decorate, decoration={markings,mark=at position .5 with {\arrow[black]{stealth}}}}];
\draw  plot[smooth, tension=1] coordinates {(v1) (-2,4.5)  (-3,3.5) (v2)}[postaction={decorate, decoration={markings,mark=at position .5 with {\arrow[black]{stealth}}}}];

\draw  (-3,7.5) -- (-1.5,5.5)[postaction={decorate, decoration={markings,mark=at position .5 with {\arrow[black]{stealth}}}}];
\draw  (-2,7.5) -- (-1.5,5.5)[postaction={decorate, decoration={markings,mark=at position .5 with {\arrow[black]{stealth}}}}];
\draw  (-0.5,7.5) -- (-1.5,5.5)[postaction={decorate, decoration={markings,mark=at position .5 with {\arrow[black]{stealth}}}}];

\draw  (-4.8,7.5)-- (-4,3)[postaction={decorate, decoration={markings,mark=at position .5 with {\arrow[black]{stealth}}}}];
\draw  (-1.5,5.5)  -- (-1.5,1)[postaction={decorate, decoration={markings,mark=at position .5 with {\arrow[black]{stealth}}}}];
\draw  (-4,3) -- (-1.5,1)[postaction={decorate, decoration={markings,mark=at position .5 with {\arrow[black]{stealth}}}}];

\draw (1,7.5)--(1.5,5.5)[postaction={decorate, decoration={markings,mark=at position .5 with {\arrow[black]{stealth}}}}];
\draw  (2.5,7.5) -- (1.5,5.5)[postaction={decorate, decoration={markings,mark=at position .5 with {\arrow[black]{stealth}}}}];
\draw  (1.5,5.5) -- (-1.5,1)[postaction={decorate, decoration={markings,mark=at position .5 with {\arrow[black]{stealth}}}}];
\draw  (-4,3) -- (-4.5,-1)[postaction={decorate, decoration={markings,mark=at position .5 with {\arrow[black]{stealth}}}}];
\draw  (-1.5,1) -- (-2,-1)[postaction={decorate, decoration={markings,mark=at position .65 with {\arrow[black]{stealth}}}}];
\draw  (0,-1) -- (-1.5,1)[postaction={decorate, decoration={markings,mark=at position .5 with {\arrowreversed[black]{stealth}}}}];
\draw  (1.5,5.5) -- (2,1.5)[postaction={decorate, decoration={markings,mark=at position .5 with {\arrow[black]{stealth}}}}];
\draw  (2,1.5) -- (2,-1)[postaction={decorate, decoration={markings,mark=at position .5 with {\arrow[black]{stealth}}}}];

\node (v17) at (6.5,7.5) {};
\node (v16) at (6.5,-1) {};
\draw[fill] (v16) circle [radius=0.15];
\draw[fill] (v17) circle [radius=0.15];
\draw  (6.5,-1) -- (6.5,7.5)[postaction={decorate, decoration={markings,mark=at position .5 with {\arrowreversed[black]{stealth}}}}];
\node (v18) at (4.5,7.5) {};
\node (v19) at (4.5,-1) {};
\node (v20) at (4.5,3.25) {};
\draw[fill] (v18) circle [radius=0.15];
\draw[fill] (v19) circle [radius=0.15];

\draw[fill] (v20) circle [radius=0.15];
\draw  (4.5,-1) -- (4.5,3.25)[postaction={decorate, decoration={markings,mark=at position .5 with {\arrowreversed[black]{stealth}}}}];
\draw  (4.5,3.24) -- (4.5,7.5)[postaction={decorate, decoration={markings,mark=at position .5 with {\arrowreversed[black]{stealth}}}}];

\draw [loosely dashed] (-6.5,7.5)--(-6.5,-1);
\draw [loosely dashed] (8,7.5)--(8,-1);
\draw [loosely dashed] (-6.5,7.5)--(8,7.5);
\draw [loosely dashed] (-6.5,-1)--(8,-1);

\end{tikzpicture}
\end{matrix}&\Longleftrightarrow&
\begin{matrix}
\begin{tikzpicture}[scale=0.3]
\node (v2) at (-4,3) {};
\draw[fill] (-1.5,5.5) circle [radius=0.15];
\node (v1) at (-1.5,5.5) {};
\node (v7) at (-1.5,1) {};
\node (v9) at (1.5,5.5) {};
\node (v14) at (2,1.5) {};
\node (v3) at (-3,7.5) {};
\node (v4) at (-2,7.5) {};
\node (v5) at (-0.5,7.5) {};
\node (v6) at (-4.8,7.5) {};
\node (v11) at (-4.5,-1) {};
\node (v12) at (-2,-1) {};
\node (v13) at (0,-1) {};
\node (v15) at (2,-1) {};
\node (v8) at (1,7.5) {};
\node (v10) at (2.5,7.5) {};
\node  at (-2.5,3.5) {};
\node  at (-3,5.2) {};
\node  at (-1.2,3.3) {};
\node  at (0.5,3.25) {};
\node  at (2.2,3.7) {};
\node  at (-3,1.7) {};
\draw[fill] (-4,3) circle [radius=0.15];
\draw[fill] (v1) circle [radius=0.15];
\draw[fill] (v7) circle [radius=0.15];
\draw[fill] (v9) circle [radius=0.15];
\draw[fill] (v14) circle [radius=0.15];
\draw[fill] (v1) circle [radius=0.15];
\draw[fill] (v2) circle [radius=0.15];

\draw  plot[smooth, tension=1] coordinates {(v1) (-2.5,5)  (-3.5,4) (v2)}[postaction={decorate, decoration={markings,mark=at position .5 with {\arrow[black]{stealth}}}}];
\draw  plot[smooth, tension=1] coordinates {(v1) (-2,4.5)  (-3,3.5) (v2)}[postaction={decorate, decoration={markings,mark=at position .5 with {\arrow[black]{stealth}}}}];

\draw  (-3,7.5) -- (-1.5,5.5)[postaction={decorate, decoration={markings,mark=at position .5 with {\arrow[black]{stealth}}}}];
\draw  (-2,7.5) -- (-1.5,5.5)[postaction={decorate, decoration={markings,mark=at position .5 with {\arrow[black]{stealth}}}}];
\draw  (-0.5,7.5) -- (-1.5,5.5)[postaction={decorate, decoration={markings,mark=at position .5 with {\arrow[black]{stealth}}}}];

\draw  (-4.8,7.5)-- (-4,3)[postaction={decorate, decoration={markings,mark=at position .5 with {\arrow[black]{stealth}}}}];
\draw  (-1.5,5.5)  -- (-1.5,1)[postaction={decorate, decoration={markings,mark=at position .5 with {\arrow[black]{stealth}}}}];
\draw  (-4,3) -- (-1.5,1)[postaction={decorate, decoration={markings,mark=at position .5 with {\arrow[black]{stealth}}}}];

\draw (1,7.5)--(1.5,5.5)[postaction={decorate, decoration={markings,mark=at position .5 with {\arrow[black]{stealth}}}}];
\draw  (2.5,7.5) -- (1.5,5.5)[postaction={decorate, decoration={markings,mark=at position .5 with {\arrow[black]{stealth}}}}];
\draw  (1.5,5.5) -- (-1.5,1)[postaction={decorate, decoration={markings,mark=at position .5 with {\arrow[black]{stealth}}}}];
\draw  (-4,3) -- (-4.5,-1)[postaction={decorate, decoration={markings,mark=at position .5 with {\arrow[black]{stealth}}}}];
\draw  (-1.5,1) -- (-2,-1)[postaction={decorate, decoration={markings,mark=at position .65 with {\arrow[black]{stealth}}}}];
\draw  (0,-1) -- (-1.5,1)[postaction={decorate, decoration={markings,mark=at position .5 with {\arrowreversed[black]{stealth}}}}];
\draw  (1.5,5.5) -- (2,1.5)[postaction={decorate, decoration={markings,mark=at position .5 with {\arrow[black]{stealth}}}}];
\draw  (2,1.5) -- (2,-1)[postaction={decorate, decoration={markings,mark=at position .5 with {\arrow[black]{stealth}}}}];

\node (v17) at (6.5,7.5) {};
\node (v16) at (6.5,-1) {};

\draw  (6.5,-1) -- (6.5,7.5)[postaction={decorate, decoration={markings,mark=at position .5 with {\arrowreversed[black]{stealth}}}}];
\node (v18) at (4.5,7.5) {};
\node (v19) at (4.5,-1) {};
\node (v20) at (4.5,3.25) {};

\draw[fill] (v20) circle [radius=0.15];
\draw  (4.5,-1) -- (4.5,3.25)[postaction={decorate, decoration={markings,mark=at position .5 with {\arrowreversed[black]{stealth}}}}];
\draw  (4.5,3.24) -- (4.5,7.5)[postaction={decorate, decoration={markings,mark=at position .5 with {\arrowreversed[black]{stealth}}}}];

\node at (-6.5,7.5) {};
\node at (-6.5,-1) {};
\node at (8,-1) {};
\node at (8,7.5) {};

\draw [loosely dashed] (-6.5,7.5)--(-6.5,-1);
\draw [loosely dashed] (8,7.5)--(8,-1);
\draw [loosely dashed] (-6.5,7.5)--(8,7.5);
\draw [loosely dashed] (-6.5,-1)--(8,-1);

\node [above] (v23) at (0.5,10.5) {$s$};
\node [below] (v33) at (0.5,-3.5){$t$};
\draw  (0.5,10.5) -- (-3,7.5);
\draw  (0.5,10.5) --(-2,7.5) ;
\draw  (0.5,10.5) --(-0.5,7.5);
\draw  (0.5,10.5) -- (-4.8,7.5);
\draw  (0.5,10.5) -- (1,7.5);
\draw (0.5,10.5) -- (2.5,7.5);
\draw  (0.5,10.5) --(6.5,7.5) ;
\draw  (0.5,10.5) -- (4.5,7.5);

\draw  (-4.5,-1) -- (0.5,-3.5);
\draw (-2,-1) -- (0.5,-3.5);
\draw (0,-1) -- (0.5,-3.5);
\draw  (2,-1) -- (0.5,-3.5);
\draw (6.5,-1) -- (0.5,-3.5);
\draw  (4.5,-1) -- (0.5,-3.5);

\draw[fill] (0.5,-3.5) circle [radius=0.15];
\draw[fill] (0.5,10.5) circle [radius=0.15];
\draw  plot[smooth, tension=.7] coordinates { (0.5,10.5) (-6.5,8.5) (-9,5.5) (-9,1) (-6.5,-2) (0.5,-3.5)}[postaction={decorate, decoration={markings,mark=at position .5 with {\arrow[black]{stealth}}}}];
\node at (-11,3) {$e$};
\end{tikzpicture}
\end{matrix}
\end{matrix}
$$
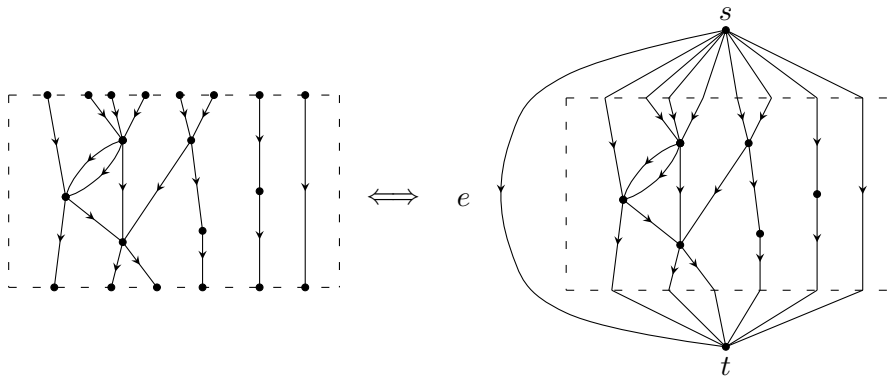
Figure $3$. A PPG and its associated upward plane $st$ graph.
\end{center}

The notion of a \emph{progressive plane graph}, with that of a PPG as a special case, was introduced by  Joyal and Street in \cite{[JS91]} as  a graphical tool for tensor calculus in monoidal categories. Similar to the above characterization of Di Battista, Tamassia \cite{[BT88]} and Kelly \cite{[Ke87]}, any progressive plane graph can be extended (in a non-unique way) to a PPG, as shown in Fig. $4$. It is clear that a progressive plane graph is essentially  an upward plane graph (possibly with isolated vertices).
\begin{center}
$$
\begin{matrix}\begin{matrix}
\begin{tikzpicture}[scale=0.35]

\draw [dashed] (-3,1.5) rectangle (8.5,-6.5);
\node [above](v1) at (-1,1.5) {};
\node (v2) at (0,-1) {};
\node (v6) at (-1.5,-4.5) {};
\node [left] at (-1.5,-4.5) {};
\node (v3) at (1.5,0.5) {};
\node (v8) at (4,-2) {};
\node (v4) at (0.5,-3.5) {};
\node [below](v5) at (0.5,-6.5) {};
\node [above](v7) at (4,1.5) {};
\draw  (-1,1.5)  -- (0,-1)[postaction={decorate, decoration={markings,mark=at position .5 with {\arrow[black]{stealth}}}}];
\draw   (1.5,0.5) -- (0,-1)[postaction={decorate, decoration={markings,mark=at position .5 with {\arrow[black]{stealth}}}}];
\draw   (1.5,0.5)--  (0.5,-3.5)[postaction={decorate, decoration={markings,mark=at position .5 with {\arrow[black]{stealth}}}}];
\draw  (0.5,-3.5) -- (v5)[postaction={decorate, decoration={markings,mark=at position .5 with {\arrow[black]{stealth}}}}];
\draw  (0,-1) -- (-1.5,-4.5)[postaction={decorate, decoration={markings,mark=at position .5 with {\arrow[black]{stealth}}}}];
\draw  (0,-1) -- (0.5,-3.5)[postaction={decorate, decoration={markings,mark=at position .5 with {\arrow[black]{stealth}}}}];
\draw  (4,1.5)-- (4,-2)[postaction={decorate, decoration={markings,mark=at position .5 with {\arrow[black]{stealth}}}}];
\draw   (1.5,0.5)-- (4,-2)[postaction={decorate, decoration={markings,mark=at position .5 with {\arrow[black]{stealth}}}}];
\node (v9) at (3,-5.5) {};
\draw  (0.5,-3.5) -- (3,-5.5)[postaction={decorate, decoration={markings,mark=at position .5 with {\arrow[black]{stealth}}}}];
\draw  (4,-2) -- (3,-5.5)[postaction={decorate, decoration={markings,mark=at position .5 with {\arrow[black]{stealth}}}}];
\node (v10) at (2,-3.5) {};
\draw   (1.5,0.5) -- (2,-3.5)[postaction={decorate, decoration={markings,mark=at position .5 with {\arrow[black]{stealth}}}}];
\draw  (4,-2) -- (2,-3.5)[postaction={decorate, decoration={markings,mark=at position .5 with {\arrow[black]{stealth}}}}];

\draw [fill](v2) circle [radius=0.11];
\draw [fill](v3) circle [radius=0.11];
\draw [fill](v4) circle [radius=0.11];
\draw [fill](v6) circle [radius=0.11];
\draw [fill](v8) circle [radius=0.11];
\draw [fill](v9) circle [radius=0.11];
\draw [fill](v10) circle [radius=0.11];

\node at (-1,0.5) {};
\node at (4.5,0) {};
\node at (0,-5.5) {};
\node at (-1.5,-3) {};

\draw(6.5,1.5)--(6.5,-6.5)[postaction={decorate, decoration={markings,mark=at position .5 with {\arrow[black]{stealth}}}}];
\node [above]at(6.5,1.5) {};
\node [below]at (6.5,-6.5) {};
\node at (7,-3) {};

\draw [fill] (-1,1.5) circle [radius=0.11];
\draw [fill](0.5,-6.5) circle [radius=0.11];
\draw [fill](4,1.5) circle [radius=0.11];
\draw [fill](6.5,-6.5) circle [radius=0.11];
\draw [fill](6.5,1.5) circle [radius=0.11];

\draw [fill](5.5,-1) circle [radius=0.11];
\draw [fill](2.7,-1.7) circle [radius=0.11];

\draw [fill](0.35,0.45) circle [radius=0.11];
\end{tikzpicture}
\end{matrix}&&&&\begin{matrix}
  \begin{tikzpicture}[scale=0.35]

\draw [loosely dashed] (-3,1.5) rectangle (8.5,-6.5);
\node [above](v1) at (-1,1.5) {};
\node (v2) at (0,-1) {};
\node (v6) at (-1.5,-4.5) {};
\node [left] at (-1.5,-4.5) {};
\node (v3) at (1.5,0.5) {};
\node (v8) at (4,-2) {};
\node (v4) at (0.5,-3.5) {};
\node [below](v5) at (0.5,-6.5) {};
\node [above](v7) at (4,1.5) {};
\draw  (-1,1.5)  -- (0,-1)[postaction={decorate, decoration={markings,mark=at position .5 with {\arrow[black]{stealth}}}}];
\draw   (1.5,0.5) -- (0,-1)[postaction={decorate, decoration={markings,mark=at position .5 with {\arrow[black]{stealth}}}}];
\draw   (1.5,0.5)--  (0.5,-3.5)[postaction={decorate, decoration={markings,mark=at position .5 with {\arrow[black]{stealth}}}}];
\draw  (0.5,-3.5) -- (v5)[postaction={decorate, decoration={markings,mark=at position .5 with {\arrow[black]{stealth}}}}];
\draw  (0,-1) -- (-1.5,-4.5) node (v19) {}[postaction={decorate, decoration={markings,mark=at position .5 with {\arrow[black]{stealth}}}}];
\draw  (0,-1) node (v16) {} -- (0.5,-3.5)[postaction={decorate, decoration={markings,mark=at position .5 with {\arrow[black]{stealth}}}}];
\draw  (4,1.5)-- (4,-2)[postaction={decorate, decoration={markings,mark=at position .5 with {\arrow[black]{stealth}}}}];
\draw   (1.5,0.5)-- (4,-2)[postaction={decorate, decoration={markings,mark=at position .5 with {\arrow[black]{stealth}}}}];
\node (v9) at (3,-5.5) {};
\draw  (0.5,-3.5) -- (3,-5.5)[postaction={decorate, decoration={markings,mark=at position .5 with {\arrow[black]{stealth}}}}];
\draw  (4,-2) -- (3,-5.5) node (v22) {}[postaction={decorate, decoration={markings,mark=at position .5 with {\arrow[black]{stealth}}}}];
\node (v10) at (2,-3.5) {};
\draw   (1.5,0.5) node (v18) {} -- (2,-3.5)[postaction={decorate, decoration={markings,mark=at position .5 with {\arrow[black]{stealth}}}}];
\draw  (4,-2) -- (2,-3.5) node (v21) {}[postaction={decorate, decoration={markings,mark=at position .5 with {\arrow[black]{stealth}}}}];

\draw [fill](v2) circle [radius=0.11];
\draw [fill](v3) circle [radius=0.11];
\draw [fill](v4) circle [radius=0.11];
\draw [fill](v6) circle [radius=0.11];
\draw [fill](v8) circle [radius=0.11];
\draw [fill](v9) circle [radius=0.11];
\draw [fill](v10) circle [radius=0.11];

\node at (-1,0.5) {};
\node at (4.5,0) {};
\node at (0,-5.5) {};
\node at (-1.5,-3) {};

\draw(6.5,1.5)--(6.5,-6.5)[postaction={decorate, decoration={markings,mark=at position .5 with {\arrow[black]{stealth}}}}];
\node [above]at(6.5,1.5) {};
\node [below]at (6.5,-6.5) {};
\node at (7,-3) {};

\draw [fill] (-1,1.5) circle [radius=0.11];
\draw [fill](0.5,-6.5) circle [radius=0.11];
\draw [fill](4,1.5) circle [radius=0.11];
\draw [fill](6.5,-6.5) circle [radius=0.11];
\draw [fill](6.5,1.5) circle [radius=0.11];

\draw [fill](5.5,-1) circle [radius=0.11];
\draw [fill](2.7,-1.7) node (v24) {} circle [radius=0.11];

\draw [fill](0.35,0.45) node (v15) {} circle [radius=0.11];

\node (v11) at (5.5,1.5) {};
\draw [fill](5.5,1.5) circle [radius=0.11];
\node (v12) at (5.5,-1) {};

\node (v13) at (5.5,-6.5) {};
\draw [fill](5.5,-6.5) circle [radius=0.11];
\draw [dashed] (5.5,1.5) --(5.5,-1)[postaction={decorate, decoration={markings,mark=at position .5 with {\arrow[black]{stealth}}}}];
\draw  [dashed](5.5,-1) -- (5.5,-6.5)[postaction={decorate, decoration={markings,mark=at position .5 with {\arrow[black]{stealth}}}}];
\node (v14) at (0,1.5) {};
\draw [fill](v14) circle [radius=0.11];
\draw  [dashed](0,1.5)-- (0.35,0.45)[postaction={decorate, decoration={markings,mark=at position .55 with {\arrow[black]{stealth}}}}];
\draw  [dashed](0.35,0.45)-- (0,-1)[postaction={decorate, decoration={markings,mark=at position .65 with {\arrow[black]{stealth}}}}];
\node (v17) at (1.5,1.5) {};
\draw [fill](v17) circle [radius=0.11];

\draw[dashed]  (1.5,1.5)-- (1.5,0.5)[postaction={decorate, decoration={markings,mark=at position .65 with {\arrow[black]{stealth}}}}];
\node (v20) at (-1.5,-6.5) {};
\draw [fill](v20) circle [radius=0.11];
\draw [dashed] (-1.5,-4.5) -- (-1.5,-6.5)[postaction={decorate, decoration={markings,mark=at position .5 with {\arrow[black]{stealth}}}}];
\draw [dashed] (2,-3.5)-- (3,-5.5)[postaction={decorate, decoration={markings,mark=at position .5 with {\arrow[black]{stealth}}}}];
\node (v23) at (3,-6.5) {};
\draw [fill](v23) circle [radius=0.11];
\draw  [dashed](3,-5.5) -- (3,-6.5)[postaction={decorate, decoration={markings,mark=at position .65 with {\arrow[black]{stealth}}}}];
\draw [dashed] (2.7,-1.7) -- (2,-3.5)[postaction={decorate, decoration={markings,mark=at position .5 with {\arrow[black]{stealth}}}}];
\draw [dashed] (1.5,0.5) -- (2.7,-1.7)[postaction={decorate, decoration={markings,mark=at position .65 with {\arrow[black]{stealth}}}}];
\end{tikzpicture}
\end{matrix}
\end{matrix}
$$

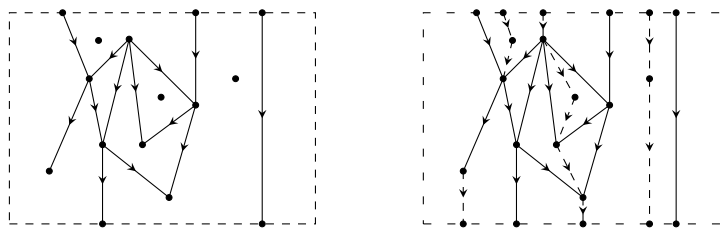
Figure $4$. A progressive plane graph and one of its PPG-extentions.

\end{center}

One main result in \cite{[HLY16]} is that a PPG can be  characterized in term of the notion of a \emph{planar order}, which is a special linear extension of the edge poset (short for \emph{partially ordered set}) of the underlying directed graph (Theorem \ref{T1}).
Based on and to generalize this result, we will give a similar characterization of upward planar graphs. Precisely, we introduce the notion of an \emph{upward planar order} (Definition \ref{upo}) for an acyclic directed graph $G$, which is a generalization of that of a planar order, and show that $G$ is upward planar if and only if it has an upward planar order (Theorem \ref{T3}). We summarize the conceptual framework as follows.

\begin{center}
\begin{tikzpicture}[scale=1.38]
\draw  (-4.5,2) rectangle (5.5,-1);
\draw  (-4.5,1) rectangle (-4.5,1) node (v1) {};
\draw  (5.5,1) rectangle (5.5,1) node (v2) {};
\draw  (v1) edge (v2);
\node (v3) at (-4.5,0) {};
\node (v4) at (5.5,0) {};
\draw  (v3) edge (v4);
\node (v7) at (1.5,2) {};
\node (v8) at (1.5,-1) {};
\node (v5) at (-1.5,2) {};
\node (v6) at (-1.5,-1) {};
\draw  (v5) edge (v6);
\draw  (v7) edge (v8);
\node at (-3,1.5) {category theory};
\node at (0,1.5) {graph theory};
\node at (-3,0.5) {PPG};
\node at (0,0.5) {upward plane $st$ graph};
\node at (-3,-0.5) {progressive plane graph};
\node at (0,-0.5) {upward plane graph};
\node at (3.5,1.5) {combinatorial characterization};
\node at (3.5,-0.5) { upward planar order};
\node at (3.5,0.5) {planar order};
\end{tikzpicture}

Figure $5$.
\end{center}

To prove our main result, Theorem \ref{T3}, we  put in much effort to analysis the relationship between the notion of a planar order and that of an upward planar order (Theorem \ref{PP and AUP equivalence}), and especially introduce the notion of a \emph{canonical processive planar extension} (\emph{CPP-extension} for short, Definition \ref{cpp}), which is the crucial link between our combinatorial characterization of PPGs and that of upward plane graphs. Our strategy of justifying our characterization of upward planar graphs by means of CPP extensions and combinatorial characterization of processive plane graphs is a combinatorial formulation of  PPG-extensions of progressive plane graphs and in a sense, a refinement of the work of Di Battista, Tamassia \cite{[BT88]} and Kelly \cite{[Ke87]}.

One advantage of our approach to study upward planarity is that it admits a composition theory of upward planar orders just as that of planar orders in \cite{[HLY16]}, which provides a practical way to compute an associated upward planar order of an upward plane graph. The composition theory will be presented elsewhere.
Another advantage is that our characterization  sheds some light on the longstanding problem \cite{[Tr78], [Ri93]} of finding a topological theory of posets (parallel to \emph{topological graph theory}) which should generalize upward planarity to higher genus surfaces. Observe that any linear extension of the edge poset of a connected acyclic directed graph $G$ naturally induces a rotation system on $G$, or a cellular embedding of $G$ on a surface, whose genus is called the \emph{linear genus} of the linear extension. The linear genus of $G$ is defined as the minimal linear genus of linear extensions of the edge poset of $G$. Our characterization means that $G$ is upward planar implies that the linear genus of $G$ is zero. However, the converse is not true; see Fig. $6$.
\begin{center}
\begin{tikzpicture}[scale=1]
\node (v6) at (-4,0.5) {};
\node (v5) at (-1,2) {};
\node (v4) at (-1,-1) {};
\node (v2) at (-2,0.5) {};
\node (v1) at (-1,0.5) {};
\node (v3) at (0,0.5) {};
\draw[fill] (v1) circle [radius=0.07];
\draw[fill] (v2) circle [radius=0.07];
\draw[fill] (v3) circle [radius=0.07];
\draw[fill] (v4) circle [radius=0.07];
\draw[fill] (v5) circle [radius=0.07];
\draw[fill] (v6) circle [radius=0.07];
\draw (-1,0.5) --(-2,0.5)[postaction={decorate, decoration={markings,mark=at position .5 with {\arrowreversed[black]{stealth}}}}];
\draw  (-1,0.5) -- (0,0.5)[postaction={decorate, decoration={markings,mark=at position .5 with {\arrowreversed[black]{stealth}}}}];
\draw (0,0.5) -- (-1,-1) [postaction={decorate, decoration={markings,mark=at position .5 with {\arrowreversed[black]{stealth}}}}];
\draw  (-2,0.5) -- (-1,-1) [postaction={decorate, decoration={markings,mark=at position .5 with {\arrowreversed[black]{stealth}}}}];
\draw   (-1,2) -- (-4,0.5)[postaction={decorate, decoration={markings,mark=at position .5 with {\arrowreversed[black]{stealth}}}}];
\draw  (-1,-1)  -- (-4,0.5)[postaction={decorate, decoration={markings,mark=at position .5 with {\arrowreversed[black]{stealth}}}}];
\draw  (-2,0.5)--  (-1,2)[postaction={decorate, decoration={markings,mark=at position .5 with {\arrowreversed[black]{stealth}}}}];
\draw (0,0.5) -- (-1,2)[postaction={decorate, decoration={markings,mark=at position .5 with {\arrowreversed[black]{stealth}}}}];
\node at (-4.3,0.7) {$s$};
\node at (-1,0.8) {$t$};
\node [scale=0.8]at (-2.8,1.4) {$1$};
\node[scale=0.8] at (-2.8,-0.4) {$2$};
\node [scale=0.8]at (-1.3,1.1) {$3$};
\node[scale=0.8] at (-0.7,1.1) {$4$};
\node [scale=0.8]at (-1.3,-0.2) {$5$};
\node [scale=0.8]at (-0.7,-0.2) {$6$};
\node [scale=0.8]at (-1.5,0.3) {$7$};
\node [scale=0.8]at (-0.6,0.3) {$8$};
\end{tikzpicture}

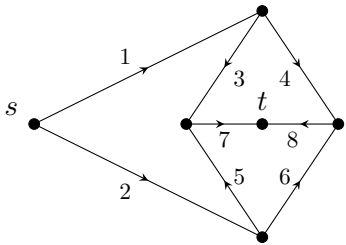
Figure $6$. A directed graph which is not upward planar but with linear genus zero.
\end{center}

It would be interesting to find possible relations between the aforementioned theory of linear genus and the higher genus theory of upward planarity proposed in \cite{[HLY16]}, which is based on the graphical calculus for symmetric monoidal categories (Chapter $2$ of \cite{[JS91]}). Yet another interesting fact is that the axioms in the definition of an upward planar order can be restated as axioms for hypergraphs, or in other words, the notion of an upward planar order makes sense for hypergraphs. These directions are worth studying in the future.

The paper is organized as follows. In Section $2$, we review the combinatorial characterization of PPGs in \cite{[HLY16]}. In Section $3$, we introduce the notions of an upward planar order and a UPO-graph and study their basic properties. In Section $4$, we give several new characterizations of PPGs. In Section $5$, we introduce the notion of a CPP-extension for an acyclic directed graph $G$. We show that there is a natural bijection between upward planar orders on $G$ and  CPP-extensions of $G$. In Section $6$, we justify the notion of a UPO-graph by showing that a UPO-graph has a unique upward planar drawing up to planar isotopy, and conversely, there is at least one upward planar order for any upward plane graph. We point out that our characterization of upward planarity can also be applied to characterize (non-directed) planar graphs.

\section{PPG and POP-graph}
In this section, we recall the notion of a PPG and its combinatorial characterization.
\begin{defn}\label{ppg}
A \emph{processive plane graph}, or PPG, is an acyclic directed graph drawn in a plane box such that $(1)$ all edges monotonically decrease in the vertical direction; $(2)$ all sources and sinks are of degree one; and $(3)$ all sources and sinks are drawn on the horizontal boundaries of the plane box.
\end{defn}
The left of Fig. $3$ shows an example. The following notion characterizes the underlying graph of a PPG.
\begin{defn}\label{pro}
A \emph{processive graph} is an acyclic directed graph with all sources and sinks being of degree one.
\end{defn}
Clearly, this notion is essentially equivalent to that of a \emph{PERT-graph} \cite{[TT86]} which is a directed graph with exactly one source $s$ and exactly one sink $t$ (the underlying graph in Fig. $6$ is an example), and also equivalent to that of an \emph{$st$ graph} which is a PERT-graph with a distinguished edge $e$ connecting $s$ and $t$ (the underlying directed graph in the right of Fig. $3$ is an example).

A vertex is called \emph{processive} if it is neither a source nor a sink.
An edge of a processive graph is called an \emph{input edge} if it starts from a source and \emph{output edge} if it ends with a sink. We denote the set of input edges of a processive graph $G$ by $I(G)$, and the set of output edges by $O(G)$.

A planar drawing of a processive graph $G$ is  \emph{boxed} if it is drawn in a plane box with all sources of $G$ on one of the horizontal boundaries of the plane box and all sinks of $G$  on the other one. From the left of Fig. $3$, it is easy to see that a PPG is a boxed and upward planar drawing of a processive graph. Two PPGs are equivalent if they are connected by a planar isotopy such that each intermediate planar drawing is boxed (but not necessarily upward).

For a vertex $v$ of a directed graph $G$, a \emph{polarization} \cite{[JS91]} of $v$ consists of two linear orders, one on the set $I_G(v)$ (or $I(v)$ for simplicity) of incoming edges of $v$ and the other on the set $O_G(v)$ (or $O(v)$ for simplicity) of outgoing edges of $v$ (possibly one of them is empty). A directed graph is called \emph{polarized} if each vertex is equipped with a polarization. In the way shown in Fig. $7$, PPGs and general upward plane graphs are polarized.

\begin{center}
\begin{tikzpicture}[scale=0.4]
\node (v2) at (0,0.5) {};
\draw[fill] (v2) circle [radius=0.13];
\node [above] (v1) at (-1.5,2.5) {};
\node [above](v3) at (-0.5,2.5) {};
\node [scale=0.8, below]at (0.5,2.5) {$\cdots$};
\node [above](v4) at (1.5,2.5) {};
\node [below](v5) at (-1.5,-1.5) {};
\node [below](v6) at (-0.5,-1.5) {};
\node[scale=.8, above](v7) at (0.5,-1.5) {$\cdots$};
\node [below](v8) at (1.5,-1.5) {};
\draw  (-1.5,2.5) -- (0,0.5)[postaction={decorate, decoration={markings,mark=at position .5 with {\arrow[black]{stealth}}}}];
\draw  (-0.5,2.5) -- (0,0.5)[postaction={decorate, decoration={markings,mark=at position .5 with {\arrow[black]{stealth}}}}];
\draw  (1.5,2.5) -- (0,0.5)[postaction={decorate, decoration={markings,mark=at position .5 with {\arrow[black]{stealth}}}}];
\draw  (0,0.5) -- (-1.5,-1.5)[postaction={decorate, decoration={markings,mark=at position .55 with {\arrow[black]{stealth}}}}];
\draw  (0,0.5) -- (-0.5,-1.5)[postaction={decorate, decoration={markings,mark=at position .55 with {\arrow[black]{stealth}}}}];
\draw  (0,0.5) -- (1.5,-1.5)[postaction={decorate, decoration={markings,mark=at position .55 with {\arrow[black]{stealth}}}}];
\node[scale=0.8] at (-1.6,1.9) {$1$};
\node [scale=0.8]at (-0.8,2.1) {$2$};
\node [scale=0.8]at (1.5,1.9) {$m$};
\node [scale=0.8]at (-1.5,-0.8) {$1$};
\node [scale=0.8]at (-0.7,-1) {$2$};
\node [scale=0.8]at (1.4,-0.9) {$n$};
\end{tikzpicture}

Figure $7$. Polarization of a vertex.
\end{center}

The following is a key notion in \cite{[HLY16]}.

\begin{defn}
A \emph{planar order} on a processive graph $G$ is a linear order $\prec$ on the edge set $E(G)$, such that

$(P_1)$ $e_1\rightarrow e_2$ implies that $e_1\prec e_2$;

$(P_2)$ if $e_1\prec e_2\prec e_3$ and $e_1\rightarrow e_3$, then either $e_1\rightarrow e_2$ or $e_2\rightarrow e_3$,

where $e_1\rightarrow e_2$ denotes that there is a directed path starting from $e_1$ and ending with  $e_2$.
\end{defn}
Figure $8$ shows a simple example motivating the definition. By the linearity of $\prec$, it is easy to see that $(P_2)$ is equivalent to $(\widetilde{P_2})$: if $e_1\prec e_2\prec e_3$ and $t(e_1)=s(e_3)$, then either $e_1\rightarrow e_2$ or $e_2\rightarrow e_3$, where $s(e)$, $t(e)$ denote the starting and ending vertex of edge $e$, respectively.

\begin{center}
\begin{tikzpicture}[scale=0.6]
\node (v2) at (-2,0.5) {};
\node (v5) at (0.5,0.5) {};
\draw[fill] (0.5,0.5) circle [radius=0.11];
\node (v12) at (3,0.5) {};
\draw[fill] (3,0.5) circle [radius=0.11];
\node [scale=0.7](v1) at (-2,2) {$1$};
\node (v3) at (-2,-1) {};
\node [scale=0.7] (v4) at (-0.5,2) {$2$};
\node [scale=0.7](v6) at (0.5,2) {$3$};
\node [scale=0.7](v7) at (1.5,2) {$4$};
\node [scale=0.7](v8) at (-0.5,-1) {$5$};
\node [scale=0.7](v9) at (0.5,-1) {$6$};
\node [scale=0.7] (v10) at (1.5,-1) {$7$};
\node [scale=0.7](v11) at (2.5,2) {$8$};
\node [scale=0.7](v13) at (3.5,2) {$9$};
\node [scale=0.7](v14) at (2.5,-1) {$10$};
\node [scale=0.7] (v15) at (3.5,-1) {$11$};
\draw  (v1) -- (v3)[postaction={decorate, decoration={markings,mark=at position .45 with {\arrow[black]{stealth}}}}];
\draw  (v4)  --  (0.5,0.5)[postaction={decorate, decoration={markings,mark=at position .45 with {\arrow[black]{stealth}}}}];
\draw  (v6)  --  (0.5,0.5)[postaction={decorate, decoration={markings,mark=at position .45 with {\arrow[black]{stealth}}}}];
\draw  (v7)  --  (0.5,0.5)[postaction={decorate, decoration={markings,mark=at position .45 with {\arrow[black]{stealth}}}}];
\draw  (0.5,0.5)  --  (v8)[postaction={decorate, decoration={markings,mark=at position .75 with {\arrow[black]{stealth}}}}];
\draw  (0.5,0.5)  --  (v9)[postaction={decorate, decoration={markings,mark=at position .75 with {\arrow[black]{stealth}}}}];
\draw  (0.5,0.5)  --  (v10)[postaction={decorate, decoration={markings,mark=at position .75 with {\arrow[black]{stealth}}}}];
\draw  (v11)  --  (3,0.5)[postaction={decorate, decoration={markings,mark=at position .45 with {\arrow[black]{stealth}}}}];
\draw  (v13)  --  (3,0.5)[postaction={decorate, decoration={markings,mark=at position .45 with {\arrow[black]{stealth}}}}];
\draw  (3,0.5)  --  (v14)[postaction={decorate, decoration={markings,mark=at position .75 with {\arrow[black]{stealth}}}}];
\draw  (3,0.5)  --  (v15)[postaction={decorate, decoration={markings,mark=at position .75 with {\arrow[black]{stealth}}}}];

\node [scale=0.7](v16) at (4.5,2) {$12$};
\node [scale=0.7](v17) at (5.5,2) {$13$};
\node [scale=0.7](v18) at (4.5,-1) {};
\node [scale=0.7] (v19) at (5.5,-1) {};
\draw  (v16)  --  (v18)[postaction={decorate, decoration={markings,mark=at position .45 with {\arrow[black]{stealth}}}}];
\draw  (v17)  --  (v19)[postaction={decorate, decoration={markings,mark=at position .45 with {\arrow[black]{stealth}}}}];
\end{tikzpicture}

Figure $8$. A planar order
\end{center}

\begin{defn}
A processive graph together with a planar order is called a \emph{planarly ordered processive graph} or \emph{POP-graph} for short.
\end{defn}
The following is a key result in \cite{[HLY16]}.
\begin{thm}\label{T1}
There is a bijection between  POP-graphs and  equivalence classes of PPGs.
\end{thm}

Figure $9$ shows the corresponding POP-graph of the PPG in the left of Fig. $3$.
\begin{center}
\begin{tikzpicture}[scale=0.4]
\node (v2) at (-4,3) {};
\draw[fill] (-1.5,5.5) circle [radius=0.11];
\node (v1) at (-1.5,5.5) {};
\node (v7) at (-1.5,1) {};
\node (v9) at (1.5,5.5) {};
\node (v14) at (2,1.5) {};
\node (v3) at (-3,7.5) {};
\node (v4) at (-2,7.5) {};
\node (v5) at (-0.5,7.5) {};
\node (v6) at (-4.8,7.5) {};
\node (v11) at (-4.5,-1) {};
\node (v12) at (-2,-1) {};
\node (v13) at (0,-1) {};
\node (v15) at (2,-1) {};
\node (v8) at (1,7.5) {};
\node (v10) at (2.5,7.5) {};

\node [scale=0.8]  at (-2.5,3.5) {$6$};
\node[scale=0.8]  at (-3,5.2) {$5$};
\node[scale=0.8]  at (-1.2,3.3) {$9$};
\node[scale=0.8]  at (0.5,3.25) {$12$};
\node[scale=0.8]  at (2.2,3.7) {$15$};
\node[scale=0.8]  at (-3,1.7) {$8$};

\node[scale=0.8] [above] at (-3,7.5) {};
\node[scale=0.8] [above]  at (-2,7.5) {};
\node[scale=0.8] [above] at (-0.5,7.5) {};
\node [scale=0.8][above]  at (-4.8,7.5) {};
\node [scale=0.8][below] at (-4.5,-1) {};
\node [scale=0.8][below]at (-2,-1) {};
\node [scale=0.8][below] at (0,-1) {};
\node [scale=0.8][below] at (2,-1) {};
\node [scale=0.8][above]  at (1,7.5) {};
\node [scale=0.8][above]  at (2.5,7.5) {};

\node  at (-2.5,3.5) {};
\node  at (-3,5.2) {};
\node  at (-1.2,3.3) {};
\node  at (0.5,3.25) {};
\node  at (2.2,3.7) {};
\node  at (-3,1.7) {};
\draw[fill] (-4,3) circle [radius=0.11];
\draw[fill] (v1) circle [radius=0.11];
\draw[fill] (v7) circle [radius=0.11];
\draw[fill] (v9) circle [radius=0.11];
\draw[fill] (v14) circle [radius=0.11];
\draw[fill] (v1) circle [radius=0.11];
\draw[fill] (v2) circle [radius=0.11];
\draw[fill] (v3) circle [radius=0.11];
\draw[fill] (v4) circle [radius=0.11];
\draw[fill] (v5) circle [radius=0.11];
\draw[fill] (v6) circle [radius=0.11];
\draw[fill] (v8) circle [radius=0.11];
\draw[fill] (v10) circle [radius=0.11];
\draw[fill] (v11) circle [radius=0.11];
\draw[fill] (v12) circle [radius=0.11];
\draw[fill] (v13) circle [radius=0.11];
\draw[fill] (v15) circle [radius=0.11];

\draw  plot[smooth, tension=1] coordinates {(v1) (-2.5,5)  (-3.5,4) (v2)}[postaction={decorate, decoration={markings,mark=at position .5 with {\arrow[black]{stealth}}}}];
\draw  plot[smooth, tension=1] coordinates {(v1) (-2,4.5)  (-3,3.5) (v2)}[postaction={decorate, decoration={markings,mark=at position .5 with {\arrow[black]{stealth}}}}];

\draw  (-3,7.5) -- (-1.5,5.5)[postaction={decorate, decoration={markings,mark=at position .5 with {\arrow[black]{stealth}}}}];
\draw  (-2,7.5) -- (-1.5,5.5)[postaction={decorate, decoration={markings,mark=at position .5 with {\arrow[black]{stealth}}}}];
\draw  (-0.5,7.5) -- (-1.5,5.5)[postaction={decorate, decoration={markings,mark=at position .5 with {\arrow[black]{stealth}}}}];

\draw  (-4.8,7.5)-- (-4,3)[postaction={decorate, decoration={markings,mark=at position .5 with {\arrow[black]{stealth}}}}];
\draw  (-1.5,5.5)  -- (-1.5,1)[postaction={decorate, decoration={markings,mark=at position .5 with {\arrow[black]{stealth}}}}];
\draw  (-4,3) -- (-1.5,1)[postaction={decorate, decoration={markings,mark=at position .5 with {\arrow[black]{stealth}}}}];

\draw (1,7.5)--(1.5,5.5)[postaction={decorate, decoration={markings,mark=at position .5 with {\arrow[black]{stealth}}}}];
\draw  (2.5,7.5) -- (1.5,5.5)[postaction={decorate, decoration={markings,mark=at position .5 with {\arrow[black]{stealth}}}}];
\draw  (1.5,5.5) -- (-1.5,1)[postaction={decorate, decoration={markings,mark=at position .5 with {\arrow[black]{stealth}}}}];
\draw  (-4,3) -- (-4.5,-1)[postaction={decorate, decoration={markings,mark=at position .5 with {\arrow[black]{stealth}}}}];
\draw  (-1.5,1) -- (-2,-1)[postaction={decorate, decoration={markings,mark=at position .65 with {\arrow[black]{stealth}}}}];
\draw  (0,-1) -- (-1.5,1)[postaction={decorate, decoration={markings,mark=at position .5 with {\arrowreversed[black]{stealth}}}}];
\draw  (1.5,5.5) -- (2,1.5)[postaction={decorate, decoration={markings,mark=at position .5 with {\arrow[black]{stealth}}}}];
\draw  (2,1.5) -- (2,-1)[postaction={decorate, decoration={markings,mark=at position .5 with {\arrow[black]{stealth}}}}];

\node (v17) at (6.5,7.5) {};
\node [right][scale=0.8] at (6.5,3.5) {$19$};
\node (v16) at (6.5,-1) {};
\draw[fill] (v16) circle [radius=0.11];
\draw[fill] (v17) circle [radius=0.11];
\draw  (6.5,-1) -- (6.5,7.5)[postaction={decorate, decoration={markings,mark=at position .5 with {\arrowreversed[black]{stealth}}}}];
\node (v18) at (4.5,7.5) {};
\node [above][scale=0.8] at (4.5,7.5) {};
\node (v19) at (4.5,-1) {};
\node [below][scale=0.8] at (4.5,-1) {};
\node (v20) at (4.5,3.25) {};
\draw[fill] (v18) circle [radius=0.11];
\draw[fill] (v19) circle [radius=0.11];

\draw[fill] (v20) circle [radius=0.11];
\draw  (4.5,-1) -- (4.5,3.25)[postaction={decorate, decoration={markings,mark=at position .5 with {\arrowreversed[black]{stealth}}}}];
\draw  (4.5,3.24) -- (4.5,7.5)[postaction={decorate, decoration={markings,mark=at position .5 with {\arrowreversed[black]{stealth}}}}];

\node [scale=0.8]at (-5,6.5) {$1$};
\node [scale=0.8]at (-3,6.5) {$2$};
\node [scale=0.8]at (-1.5,7) {$3$};
\node [scale=0.8]at (-0.5,6.5) {$4$};
\node [scale=0.8]at (0.6,6.5) {$10$};
\node [scale=0.8]at (2.6,6.5) {$11$};
\node [scale=0.8]at (5,6) {$17$};
\node [scale=0.8]at (-4.7,0.5) {$7$};
\node [scale=0.8]at (-2.5,0) {$13$};
\node [scale=0.8]at (0,0) {$14$};
\node [scale=0.8]at (2.5,0) {$16$};
\node [scale=0.8]at (5,0.5) {$18$};
\end{tikzpicture}

Figure $9$. A POP-graph.
\end{center}

\section{\UPO graph}
In this section, we introduce the key notion in this paper, that is, the notion of a UPO-graph and show some of its basic properties.

We first introduce some notations. Let $S$ be a finite set with a linear order $<$.  Given a subset $X\subseteq S$, we write $X^-=\min{X}$ and $X^+=\max{X}$. The  convex hull of $X$ in $S$ is $\overline{X} = \{y\in S| X^-\leq y\leq X^+\}$.

\begin{defn}\label{upo}
An \emph{upward planar order} on a directed graph $G$ is a linear order $\prec$ on $E(G)$, such that

$(U_1)$ $e_1\rightarrow e_2$ implies that $e_1\prec e_2$;

$(U_2)$ for any vertex $v$, $\overline{I(v)}\cap \overline{O(v)}=\emptyset$ and $\overline{E(v)}=\overline{I(v)}\sqcup \overline{O(v)}$;

$(U_3)$ for any two  vertices $v_1$ and $v_2$, $I(v_1)\cap \overline{I(v_2)}\neq\emptyset$ implies that $\overline{I(v_1)}\subseteq \overline{I(v_2)}$, and $O(v_1)\cap \overline{O(v_2)}\neq\emptyset$ implies that $\overline{O(v_1)}\subseteq \overline{O(v_2)}$.
\end{defn}

Figure $10$ is a typical example motivating this definition.
\begin{center}
 \begin{tikzpicture}[scale=0.55]

\draw[dashed]  (-4,2.5) rectangle (13,-1.5);
\node [scale=0.8, above] (v1) at (-2.5,2.5) {1};
\node [scale=0.8,below] (v2) at (-2.5,-1.5) {};
\node (v3) at (1,0.5) {};
\draw [fill](v3) circle [radius=0.11];
\node [scale=0.8,below](v4) at (-0.2,-1.5) {3};
\node [scale=0.8,above](v5) at (2.5,2.5) {5};
\node[scale=0.8,above] (v7) at (4,2.5) {6};
\node (v6) at (3,0.5) {};
\draw [fill](v6) circle [radius=0.11];
\node [scale=0.8,below](v8) at (3,-1.5) {7};
\node (v9) at (6.5,1.5) {};
\draw [fill](v9) circle [radius=0.11];
\node [scale=0.8,below](v10) at (4.5,-1.5) {8};
\node [scale=0.8,below](v11) at (5.5,-1.5) {9};
\node [scale=0.8,below](v13) at (6.2,-1.5) {10};
\node[scale=0.8,below] (v14) at (7.2,-1.5) {11};
\node[scale=0.8,below] (v15) at (8,-1.5) {12};
\node [scale=0.8,above](v16) at (9,2.5) {13};
\node [scale=0.8,above](v23) at (10.6,2.5) {15};
\node [scale=0.8,above](v18) at (11.5,2.5) {16};
\node (v17) at (10.5,0.5) {};
\draw [fill](v17) circle [radius=0.11];
\node (v12) at (6.5,0) {};
\draw [fill](v12) circle [radius=0.11];

\draw  (v1) -- (v2)[postaction={decorate, decoration={markings,mark=at position .5 with {\arrow[black]{stealth}}}}];
\draw  (-0.2,0.5) -- (v4)[postaction={decorate, decoration={markings,mark=at position .5 with {\arrow[black]{stealth}}}}];
\draw  (v5) -- (3,0.5)[postaction={decorate, decoration={markings,mark=at position .5 with {\arrow[black]{stealth}}}}];
\draw  (v7) -- (3,0.5)[postaction={decorate, decoration={markings,mark=at position .5 with {\arrow[black]{stealth}}}}];
\draw  (3,0.5)-- (v8)[postaction={decorate, decoration={markings,mark=at position .5 with {\arrow[black]{stealth}}}}];
\draw  (6.5,1.5)-- (4.5,-1.5)[postaction={decorate, decoration={markings,mark=at position .5 with {\arrow[black]{stealth}}}}];
\draw  (6.5,1.5) -- (v11)[postaction={decorate, decoration={markings,mark=at position .5 with {\arrow[black]{stealth}}}}];
\draw (6.5,0) -- (v13)[postaction={decorate, decoration={markings,mark=at position .5 with {\arrow[black]{stealth}}}}];
\draw (6.5,0) -- (v14)[postaction={decorate, decoration={markings,mark=at position .5 with {\arrow[black]{stealth}}}}];
\draw  (6.5,1.5) -- (v15)[postaction={decorate, decoration={markings,mark=at position .5 with {\arrow[black]{stealth}}}}];
\node [scale=0.8,below](v19) at (9,-1.5) {17};
\node (v22) at (10.5,1.5) {};
\draw [fill](v22) circle [radius=0.11];
\node [scale=0.8,above](v21) at (9.8,2.5) {14};
\draw  (v16) -- (10.5,0.5)[postaction={decorate, decoration={markings,mark=at position .5 with {\arrow[black]{stealth}}}}];
\draw  (v18) -- (10.5,0.5)[postaction={decorate, decoration={markings,mark=at position .5 with {\arrow[black]{stealth}}}}];
\draw  ((10.5,0.5) -- (v19)[postaction={decorate, decoration={markings,mark=at position .5 with {\arrow[black]{stealth}}}}];
\node [scale=0.8,below](v20) at (12,-1.5) {21};
\draw (10.5,0.5) --(v20)[postaction={decorate, decoration={markings,mark=at position .5 with {\arrow[black]{stealth}}}}];
\draw  (v21) -- (10.5,1.5)[postaction={decorate, decoration={markings,mark=at position .5 with {\arrow[black]{stealth}}}}];
\draw  (v23) -- (10.5,1.5)[postaction={decorate, decoration={markings,mark=at position .5 with {\arrow[black]{stealth}}}}];
\node (v24) at (10.5,-0.5) {};
\draw [fill](v24) circle [radius=0.11];
\node[scale=0.8,below] (v25) at (9.8,-1.5) {18};
\node [scale=0.8,below](v26) at (10.5,-1.5) {19};
\node [scale=0.8,below](v27) at (11.2,-1.5) {20};
\draw (10.5,-0.5) -- (v25)[postaction={decorate, decoration={markings,mark=at position .65 with {\arrow[black]{stealth}}}}];
\draw  (10.5,-0.5) -- (v26)[postaction={decorate, decoration={markings,mark=at position .65 with {\arrow[black]{stealth}}}}];
\draw  (10.5,-0.5) -- (v27)[postaction={decorate, decoration={markings,mark=at position .65 with {\arrow[black]{stealth}}}}];

\draw[dashed]  (-4,2.5) rectangle (13,-1.5);
\node [scale=0.8, above] (v1) at (-2.5,2.5) {1};
\node [scale=0.8,below] (v2) at (-2.5,-1.5) {};
\node (v3) at (-0.2,0.5) {};
\draw [fill](v3) circle [radius=0.11];
\node [scale=0.8,below](v4) at (1,-1.5) {};
\node [scale=0.8,above](v5) at (2.5,2.5) {};
\node[scale=0.8,above] (v7) at (4,2.5) {};
\node (v6) at (3,0.5) {};
\draw [fill](v6) circle [radius=0.11];
\node [scale=0.8,below](v8) at (3,-1.5) {};
\node (v9) at (6.5,1.5) {};
\draw [fill](v9) circle [radius=0.11];

\node (v17) at (10.5,0.5) {};
\draw [fill](v17) circle [radius=0.11];
\node (v12) at (6.5,0) {};
\draw [fill](v12) circle [radius=0.11];

\node [scale=0.8,above](v40) at (-1.2,2.5) {2};
\node (v41) at (-1.2,0.5) {};
\draw [fill](v41) circle [radius=0.11];
\draw (-1.2,2.5) -- (-1.2,0.5)[postaction={decorate, decoration={markings,mark=at position .5 with {\arrow[black]{stealth}}}}];

\node [scale=0.8,above](v50) at (1,2.5) {4};
\node (v51) at (1,0.5) {};
\draw [fill](v51) circle [radius=0.11];
\draw (1,2.5) -- (1,0.5)[postaction={decorate, decoration={markings,mark=at position .5 with {\arrow[black]{stealth}}}}];
\end{tikzpicture}

  Figure $10$. An upward planar order.
\end{center}

\begin{defn}
A directed graph together with an upward planar order is called an \emph{upward planarly ordered graph} or \emph{\UPO graph} for short.
\end{defn}
Any \UPO graph  must be acyclic. Obviously,
 $(U_1)=(P_1)$, say $\prec$ is a linear extension of $\rightarrow$.
$(U_2)$, under $(U_1)$, is equivalent to $O(v)^-=I(v)^++1$ (with respect to $\prec$) for any processive  vertex $v$.

The following lemma is an easy consequence of $(U_3)$.

\begin{lem}\label{property of UPG}
Let $G$ be an acyclic directed graph, $\prec$ a linear order on $E(G)$, and $v_1$ and $v_2$ be two vertices of $G$. If $\prec$ satisfies $(U_3)$, then

$(1)$ $\overline{I(v_1)}\cap \overline{I(v_2)}\neq \emptyset$ implies that either $\overline{I(v_1)}\subseteq \overline{I(v_2)}$ or $\overline{I(v_2)}\subseteq \overline{I(v_1)}$.

$(2)$  $\overline{O(v_1)}\cap \overline{O(v_2)}\neq \emptyset$ implies that either $\overline{O(v_1)}\subseteq \overline{O(v_2)}$ or $\overline{O(v_2)}\subseteq \overline{O(v_1)}$.
\end{lem}
\begin{proof}
We only prove $(1)$. The proof of $(2)$ is similar. Both $\overline{I(v_1)}$ and $\overline{I(v_2)}$ are intervals of $(E(G),\prec)$, then $\overline{I(v_1)}\cap \overline{I(v_2)}\neq \emptyset$ implies that either $I(v_1)^+\in \overline{I(v_2)}$ or $I(v_2)^+\in \overline{I(v_1)}$. In the former case, notice that $I(v_1)^+\in I(v_1)$, so $I(v_1)\cap \overline{I(v_2)}\neq \emptyset$. By $(U_3)$, $\overline{I(v_1)}\subset \overline{I(v_2)}$. Similarly, in the latter case, $\overline{I(v_2)}\subset \overline{I(v_1)}$.
\end{proof}
An \emph{embedding} of directed graphs  $\phi\colon G_1\rightarrow G_2$ consists of a pair of injections $\phi_0\colon V(G_1) \rightarrow V(G_2)$  and $\phi_1\colon E(G_1)\rightarrow E(G_2)$, such that $s(\phi_1(e))=\phi_0(s(e))$ and $t(\phi_1(e))=\phi_0(t(e))$ for any $e\in E(G_1)$. In this case, $G_1$ is called a \emph{subgraph} of $G_2$. We freely identify the vertices and edges of $G_1$ with their images, and view $V(G_1)\subseteq V(G_2)$, $E(G_1)\subseteq E(G_2)$.

The following proposition shows that $(U_1)$ and $(U_3)$ are hereditary.

\begin{prop}\label{restriction}
Let $G$ be an acyclic directed graph with a linear order $\prec$ on $E(G)$, $H$ a subgraph of $G$ and $\prec_H$ the linear order on $E(H)$ induced from $\prec$. Then

$(1)$ $\prec$ satisfies $(U_1)$ implies that $\prec_H$ satisfies $(U_1)$.

$(2)$ $\prec$ satisfies $(U_3)$ implies that $\prec_H$ satisfies $(U_3)$.
\end{prop}

\begin{proof}
$(1)$ is a direct consequence of the facts that for $e_1, e_2\in E(H)$, $e_1\rightarrow e_2$ in $H$ if and only if $e_1\rightarrow e_2$ in $G$, and that $e_1\prec_H e_2$ if and only if $e_1\prec e_2$.

Now we prove $(2)$ by contradiction. Suppose there exist $v_1,v_2\in V(H)$, such that  $I_H(v_1)\cap \overline{I_H(v_2)}\neq \emptyset$ and $\overline{I_H(v_1)}\not\subseteq \overline{I_H(v_2)}$.
On one hand,  $I_H(v_1)\cap \overline{I_H(v_2)}\neq \emptyset$ implies that $I_G(v_1)\cap \overline{I_G(v_2)}\neq \emptyset$, then by $(U_3)$ of $\prec$,  $\overline{I_G(v_1)}\subseteq \overline{I_G(v_2)}$.

On the other hand, we must have $\overline{I_G(v_2)}\subseteq \overline{I_G(v_1)}$. In fact, $I_H(v_1)\cap \overline{I_H(v_2)}\neq \emptyset$ means that there is an edge $e\in I_H(v_1)$ such that $I_H(v_2)^-\preceq_H e\preceq_H I_H(v_2)^+$, and $\overline{I_H(v_1)}\not\subseteq \overline{I_H(v_2)}$  means that  there is an edge $h\in I_H(v_1)$ such that $h\prec_H I_H(v_2)^-$ or $I_H(v_2)^+\prec_H h$. In the former case, $I_H(v_2)^-\in [h,e]\subseteq \overline{I_H(v_1)}$. In the latter case, $I_H(v_2)^+\in [e,h]\subseteq \overline{I_H(v_1)}$. So in both cases, $I_H(v_2)\cap \overline{I_H(v_1)}\neq \emptyset$, and hence $I_G(v_2)\cap \overline{I_G(v_1)}\neq \emptyset$. Then by $(U_3)$ of $\prec$,  $\overline{I_G(v_2)}\subseteq \overline{I_G(v_1)}$.

In summary, $\overline{I_G(v_1)}= \overline{I_G(v_2)}$, that is, $v_1=v_2$, which contradicts the assumption $\overline{I_H(v_1)}\not\subseteq \overline{I_H(v_2)}$.
Similarly, we can show that $O_H(v_1)\cap \overline{O_H(v_2)}\neq \emptyset$ implies $\overline{O_H(v_1)}\subseteq \overline{O_H(v_2)}$.
\end{proof}

Let $(G,\prec)$ be a \UPO graph and $v$ a vertex of $G$. We set
\begin{equation*}U(v)=
\begin{cases}
\ \{w|w\in V(G),\ \overline{O(v)}\subsetneq \overline{O(w)}\},&  \text{if}\ O(v)\neq\emptyset;\\
\  \emptyset,&  \text{otherwise};\\
\end{cases}
\end{equation*}
\begin{equation*}D(v)=
\begin{cases}
\ \{w|w\in V(G),\ \overline{I(v)}\subsetneq \overline{I(w)}\},&  \text{if}\ I(v)\neq\emptyset;\\
\  \emptyset,&  \text{otherwise}.\\
\end{cases}
\end{equation*}

We define an order $<$ on $U(v)$ as follows. For any $w_1,w_2\in U(v)$, $w_1<w_2$ if $\overline{O(w_1)}\subsetneq \overline{O(w_2)}$. By Lemma \ref{property of UPG}, $<$ is a linear order on $U(v)$. Similarly, $D(v)$ is a linearly ordered set under the order that $w_1<w_2$ in $D(v)$ if $\overline{I(w_1)}\subsetneq \overline{I(w_2)}$.

The following theorem shows that when suitable vertices and edges are added to a \UPO graph, the resulting graph will admit a unique extended upward planar order.

\begin{thm} \label{L_1_ext}
Let $(G,\prec_G)$ be a \UPO graph.  $S(G)$ and $T(G)$ are the sets of sources and sinks of $G$, respectively.  Assume that  $\Gamma$ is a directed graph obtained by adding a new edge $e$ to $G$ in any one of the following ways (see Fig. $11$):

 $(1)$  $t(e)\in S(G)$, $U(t(e))= \emptyset$ in $(G,\prec_G)$, and $s(e)\not\in V(G)$;

 $(2)$  $t(e)\in S(G)$, $U(t(e))\ne \emptyset$ in $(G,\prec_G)$, and $s(e)= U(t(e))^-$;

$(3)$  $s(e)\in T(G)$, $D(s(e))= \emptyset$ in $(G,\prec_G)$, and $t(e)\not\in V(G)$;

 $(4)$  $s(e)\in T(G)$, $D(s(e))\ne \emptyset$ in $(G,\prec_G)$, and $t(e)= D(s(e))^-$.

\noindent Then there exists a unique upward planar order $\prec_\Gamma$ on $\Gamma$, whose restriction is $\prec_G$.
 \end{thm}

\begin{center}
\begin{tikzpicture}[scale=0.6]

\node (v2) at (-3.5,0) {};
\node (v3) at (-4.5,-2) {};
\node (v4) at (-2.5,-2) {};
\node(v1) at (-3.5,2.5) {};
\node [above] at (-3.5,2.5) {$s(e)\not\in V(G)$};
\draw[fill] (v2) circle [radius=0.07];
\draw[fill] (v1) circle [radius=0.07];
\draw [dashed] (-3.5,2.5) -- (-3.5,0)[postaction={decorate, decoration={markings,mark=at position .5 with {\arrow[black]{stealth}}}}];
\draw  (-3.5,0)-- (-4.5,-2)[postaction={decorate, decoration={markings,mark=at position .5 with {\arrow[black]{stealth}}}}];
\draw  (-3.5,0) -- (-2.5,-2)[postaction={decorate, decoration={markings,mark=at position .5 with {\arrow[black]{stealth}}}}];
\node at (-3.5,-3.5) {$(1)$};

\node (v8) at (0.5,0) {};
\node (v9) at (-0.5,-2.25) {};
\node (v10) at (1.5,-2.25) {};
\node (v5) at (0.5,2.5) {};
\node (v6) at (-1,1) {};
\node (v7) at (2,1) {};
\draw  (0.5,2.5)-- (v6)[postaction={decorate, decoration={markings,mark=at position .5 with {\arrow[black]{stealth}}}}];
\draw  (0.5,2.5) -- (v7)[postaction={decorate, decoration={markings,mark=at position .5 with {\arrow[black]{stealth}}}}];
\draw  (0.5,0) -- (v9)[postaction={decorate, decoration={markings,mark=at position .5 with {\arrow[black]{stealth}}}}];
\draw  (0.5,0) -- (v10)[postaction={decorate, decoration={markings,mark=at position .5 with {\arrow[black]{stealth}}}}];
\draw[fill] (v5) circle [radius=0.07];
\draw[fill] (v8) circle [radius=0.07];
\node at (-4,1.5) {$e$};
\draw [dashed] (0.5,2.5) -- (0.5,0)[postaction={decorate, decoration={markings,mark=at position .5 with {\arrow[black]{stealth}}}}];
\node at (0,1) {$e$};
\node at (0.5,-3.5) {$(2)$};

\node at (-3.5,-2) {$\cdots$};
\node at (0.5,-2) {$\cdots$};
\node at (0.5,3) {$U(t(e))^-$};
\node (v12) at (5,0.5) {};
\node (v14) at (5,-2) {};
\node (v11) at (4,2.75) {};
\node (v13) at (6,2.75) {};
\draw[fill] (v12) circle [radius=0.07];
\draw[fill] (v14) circle [radius=0.07];
\draw  (v11)-- (5,0.5)[postaction={decorate, decoration={markings,mark=at position .5 with {\arrow[black]{stealth}}}}];
\draw  (v13) -- (5,0.5)[postaction={decorate, decoration={markings,mark=at position .5 with {\arrow[black]{stealth}}}}];
\draw [dashed] (5,0.5) -- (5,-2)[postaction={decorate, decoration={markings,mark=at position .5 with {\arrow[black]{stealth}}}}];
\node at (4.5,-1) {$e$};
\node [below] at (5,-2) {$t(e)\not\in V(G)$};
\node at (5,-3.5) {$(3)$};

\node at (5,2.5) {$\cdots$};
\node (v16) at (9,0.5) {};
\node (v15) at (7.5,2.75) {};
\node (v17) at (10,2.75) {};
\node (v19) at (9,-2) {};
\node (v20) at (7.5,-0.5) {};
\node (v18) at (10,-0.5) {};
\draw  (v15) --(9,0.5) [postaction={decorate, decoration={markings,mark=at position .5 with {\arrow[black]{stealth}}}}];
\draw  (v17) -- (9,0.5) [postaction={decorate, decoration={markings,mark=at position .5 with {\arrow[black]{stealth}}}}];
\draw  (v18) -- (9,-2)[postaction={decorate, decoration={markings,mark=at position .5 with {\arrow[black]{stealth}}}}];
\draw  (v20) -- (9,-2)[postaction={decorate, decoration={markings,mark=at position .5 with {\arrow[black]{stealth}}}}];
\draw[fill] (v16) circle [radius=0.07];
\draw[fill] (v19) circle [radius=0.07];
\draw [dashed]  (9,0.5) -- (9,-2)[postaction={decorate, decoration={markings,mark=at position .5 with {\arrow[black]{stealth}}}}];
\node at (8.5,-0.5) {$e$};
\node at (9,-3.5) {$(4)$};
\node at (9,-2.5) {$D(s(e))^-$};
\node at (8.75,2.5) {$\cdots$};
\end{tikzpicture}

\begin{center}
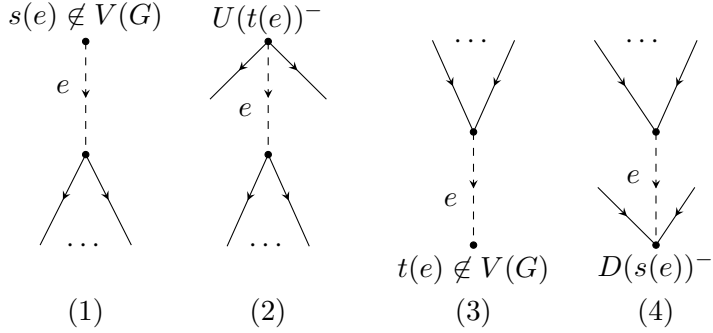
Figure $11$. Local configurations of $e$ in Theorem \ref{L_1_ext}.
\end{center}
\end{center}
 \begin{proof}
 The uniqueness follows from $(U_2)$ of $\prec_\Gamma$. In fact, in cases (1) and (2), $e=O(t(e))^--1$; and in cases (3) and (4), $e=I(s(e))^++1$.
In this way, the linear order $\prec_\Gamma$ on $E(\Gamma)=E(G)\sqcup \{e\}$ is uniquely defined. To show the existence it suffices to show that $\prec_\Gamma$ is an upward planar order.

$(i)$ First we check $(U_1)$ for $\prec_\Gamma$. In case $(1)$, $e\in I(\Gamma)$, so we only need to show that  $e\rightarrow e_1$ implies that $e\prec_\Gamma e_1$. In fact,  $e\rightarrow e_1$, implies that there is an edge $e_2\in O(t(e))$ such that $e_2\rightarrow e_1$ in $G$, so $e_2\prec_G e_1$ and hence $e_2\prec_\Gamma e_1$. Then  $e=O(t(e))^--1\prec_\Gamma e_2\prec_\Gamma e_1$.

In case $(2)$, we only need to show that $e_1\rightarrow e$ and $e\rightarrow e_2$ implies that $e_1\prec_\Gamma e_2$. Set $w=U(t(e))^-$. On one hand, $\overline{O(t(e))}\subsetneq \overline{O(w)}$ and $e=O(t(e))^--1$ implies that $O(w)^-\prec_\Gamma e$. By $e_1\rightarrow e$, $e_1\rightarrow O(w)^-$ in $G$, so $e_1\prec_G O(w)^-$ and hence $e_1\prec_\Gamma O(w)^-\prec_\Gamma e$.
On the other hand, $\overline{O(t(e))}\subsetneq \overline{O(w)}$ implies the non-existence of a direct path in $G$ that starts from $t(e)$ and ends with $w$, so just as case $(1)$, $e\rightarrow e_2$ implies that there is an edge $e_3\in O(t(e))$ such that $e_3\rightarrow  e_2$ in $G$, which implies that $e\prec_\Gamma e_3\prec_\Gamma e_2$. In summary,  $e_1\prec_\Gamma O(w)^-\prec_\Gamma e\prec_\Gamma e_3\prec_\Gamma e_2$. The proofs in cases $(3)$ and $(4)$ are similar.

$(ii)$ Under $(U_1)$, $(U_2)$ is equivalent to $O(v)^-=I(v)^++1$ which is obvious from the construction of $\prec_\Gamma$.

$(iii)$ Now we are left to check $(U_3)$.
If $v_1=v_2$ or $v_1\in V(\Gamma)-V(G)$ or  $v_2\in V(\Gamma)-V(G)$, $(U_3)$ is trivial.
So we assume $v_1$ and $v_2$ are different vertices of $G$ such that $I_\Gamma(v_1)\cap \overline{I_\Gamma(v_2)}\neq \emptyset$.  There are two possibilities for $v_1$.

If $e\not\in I_\Gamma(v_1)$, then $I_\Gamma(v_1)=I_G(v_1)$. So $I_G(v_1)\cap \overline{I_G(v_2)}= I_G(v_1)\cap (\overline{I_G(v_2)}\sqcup \{e\})\supseteq I_\Gamma(v_1)\cap \overline{I_\Gamma(v_2)}\neq\emptyset$. Applying $(U_3)$ for $\prec_G$, we have $\overline{I_G(v_1)}\subseteq \overline{I_G(v_2)}$.  Then $I_\Gamma(v_1)=I_G(v_1)\subseteq \overline{I_G(v_1)}\subseteq \overline{I_G(v_2)}\subseteq \overline{I_\Gamma(v_2)}$, which implies that $\overline{I_\Gamma(v_1)}\subseteq \overline{I_\Gamma(v_2)}$.

Otherwise, $e\in I_\Gamma(v_1)$, there are three cases.  In cases $(1)$ and $(2)$, $I_\Gamma(v_1)= \overline{I_\Gamma(v_1)}=\{e\}$ and hence $\overline{I_\Gamma(v_1)}\subseteq \overline{I_\Gamma(v_2)}$.
In case (4), $v_1=t(e)=D(s(e))^-$. Note that $I_\Gamma(v_1)=I_G(v_1)\sqcup\{e\}$, so $I_\Gamma(v_1)\cap \overline{I_\Gamma(v_2)}\neq\emptyset$ implies that $I_G(v_1)\cap \overline{I_\Gamma(v_2)}\neq \emptyset$ or $e\in \overline{I_\Gamma(v_2)}$.
In the former case, since $e\not\in I_G(v_1)$, so $I_G(v_1)\cap \overline{I_G(v_2)}\neq\emptyset$, which, by applying $(U_3)$ for $\prec_G$, implies that $\overline{I_G(v_1)}\subseteq \overline{I_G(v_2)}$, and hence $\overline{I_\Gamma(v_1)}\subseteq \overline{I_\Gamma(v_2)}$.
In the later case, note that $e=I_G(s(e))^++1$ and $e\not\in I_\Gamma(v_2)$ (by $v_1\neq v_2$), so $e\in \overline{I_\Gamma(v_2)}$ implies that $I_G(s(e))^+\in \overline{I_\Gamma(v_2)}$. So $I_G(s(e))\cap \overline{I_G(v_2)}\neq\emptyset$, which, by applying $(U_3)$ for $\prec_G$, implies that $\overline{I_G(s(e))}\subseteq \overline{I_G(v_2)}$, that is, $v_2\in D(s(e))$. Since $v_1=D(s(e))^-$, so $\overline{I_G(v_1)}\subseteq \overline{I_G(v_2)}$, which implies that $\overline{I_\Gamma(v_1)}\subseteq \overline{I_\Gamma(v_2)}$.

Dually, we can show that for any different $v_1, v_2\in V(G)$, $\overline{O_\Gamma(v_1)}\subseteq \overline{O_\Gamma(v_2)}$ provided that $O_\Gamma(v_1)\cap \overline{O_\Gamma(v_2)}\neq\emptyset$. The proof is completed.
 \end{proof}
\section{Characterizations of POP-graphs}
In this section, we introduce some constraints for UPO-graphs and list some lemmas which are useful for proving new characterizations of POP-graphs.

\begin{lem}\label{progressive}
Let $G$ be a processive graph with a linear order $\prec$ on $E(G)$. Then

$(1)$ for any vertices $v_1\ne v_2$ of $G$, $I(v_1)\cap \overline{I(v_2)}\neq\emptyset$ implies that $v_2$ is processive.

$(2)$ for any vertices $v_1\ne v_2$ of $G$, $O(v_1)\cap \overline{O(v_2)}\neq\emptyset$ implies that $v_2$ is processive.
\end{lem}
\begin{proof}
We only prove $(1)$. By assumption $v_1\neq v_2$ and $I(v_1)\cap \overline{I(v_2)}\neq\emptyset$, we know  that the degree of $v_2$ is not equal to one.  Since $G$ is a processive graph,   $v_2$ must be processive. The proof of $(2)$ is similar.
\end{proof}
\begin{defn}
A UPO-graph is called  \emph{anchored} if $(A)$ for any different vertices $v_1$, $v_2$ of $G$, $\emptyset\neq\overline{I(v_1)}\subset \overline{I(v_2)}$ implies that $v_1\rightarrow v_2$; $\emptyset\neq\overline{O(v_1)}\subset \overline{O(v_2)}$ implies that $v_2\rightarrow v_1$, where $v\rightarrow w$ denotes that there is a directed path starting from $v$ and ending with $w$.
\end{defn}
In Section $6$, we will show that any UPO-graph has an upward planar drawing, which is called the \emph{geometric realization} of the UPO-graph. In the geometric realization of an anchored UPO-graph, all sources and sinks are drawn on the boundary of the external face; see Remark \ref{Re2} for explanation.

\begin{lem}\label{property of PPG}
Let $(G,\prec)$ be a POP-graph. Then $\prec$ satisfies  $(A)$.
\end{lem}
\begin{proof}
We only prove the first part of $(A)$. The proof of the second part is similar and we omit it here.
Let $v_1\ne v_2$ be two vertices of $G$, such that $\overline{I(v_1)}\subset \overline{I(v_2)}$ and $\overline{I(v_1)}\neq \emptyset$.
Then $I(v_1)\cap \overline{I(v_2)}\neq\emptyset$, and therefore by Lemma \ref{progressive} $(1)$, $v_2$ is processive.

Take $e\in I(v_1)\subset \overline{I(v_2)}$. Then $I(v_2)^-\prec e\prec I(v_2)^+$, and hence $I(v_2)^-\prec e\prec O(v_2)^-$, where the last equality follows from $(P_1)$ and the fact $I(v_2)^+\to O(v_2)^-$. By $(P_2)$, $I(v_2)^-\rightarrow O(v_2)^-$ implies that either $I(v_2)^-\rightarrow e$ or $e\rightarrow O(v_2)^-$. If $I(v_2)^-\rightarrow e$, then $I(v_2)^+\rightarrow e$, and by $(P_1)$, $I(v_2)^+\prec e$, a contradiction. Thus we must have $e\rightarrow O(v_2)^-$, which implies that $v_1\rightarrow v_2$.
\end{proof}

\begin{lem}\label{A and U_4}
Let $(G,\prec)$ be a \UPO graph. If $G$ is a processive graph, then  $(A)$ is equivalent to the following condition:

$(U_4)$ for any processive vertex $v$ of $G$, $I(G)\cap \overline{O(v)}=\emptyset$ and $O(G)\cap \overline{I(v)}=\emptyset$.
\end{lem}

\begin{proof}
$(A)\Longrightarrow (U_4)$. We prove this by contradiction. Suppose there is a processive vertex $v$ of $G$ such that $I(G)\cap \overline{O(v)}\neq\emptyset$.  Take $i\in I(G)\cap \overline{O(v)}$ and set $w=s(i)$. Clearly, by definition, $\overline{O(w)}=\{i\}$. Then  $\overline{O(w)}\subset \overline{O(v)}$ and  $\overline{O(w)}\neq\emptyset$. Thus by $(A)$ we have $v\rightarrow w$, which contradicts  $i\in I(G)$.
Similarly, we can prove that for any processive vertex $v$ of $G$, $O(G)\cap \overline{I(v)}=\emptyset$.

$(U_4)\Longrightarrow (A)$. Let $v_1$, $v_2$ be two different vertices of $G$ with  $\overline{I(v_1)}\subset \overline{I(v_2)}$ and $\overline{I(v_1)}\neq \emptyset$, we want to prove that $v_1\rightarrow v_2$. First, $\overline{I(v_1)}\subset \overline{I(v_2)}$ and $\overline{I(v_1)}\neq \emptyset$ imply that $I(v_1)\cap \overline{I(v_2)}=I(v_1)\neq \emptyset$, and hence by Lemma \ref{progressive} $(1)$, $v_2$ is a processive vertex.

Next, we claim that $v_1$ must not be a sink. If not, by the fact that $G$ is a processive graph, $\overline{I(v_1)}=\{e\}\subseteq I(G)$, which implies that $I(G)\cap \overline{I(v_2)}\supseteq\{e\}\neq\emptyset$, contradicting $(U_4)$. So $v_1$ is a processive vertex.

Now by $(U_2)$, $\overline{I(v_1)}\subset \overline{I(v_2)}$ implies that $I(v_2)^-\prec O(v_1)^-=I(v_1)^++1\preceq I(v_2)^+$. If $O(v_1)^-=I(v_2)^+$, then $v_1\rightarrow v_2$ and we complete the proof. Otherwise, $I(v_2)^-\prec O(v_1)^-\prec I(v_2)^+$, which implies that $I(w_1)\cap \overline{I(v_2)}\neq \emptyset$,  where $w_1=t(O(v_1)^-)$.  Clearly, $v_1\rightarrow w_1$ and by $(U_3)$, $\overline{I(w_1)}\subset \overline{I(v_2)}$. If $w_1\rightarrow v_2$, then $v_1\rightarrow v_2$ and we complete the proof. Otherwise, note that $O(v_1)^-\in \overline{I(w_1)}\neq \emptyset$, similar as $v_1$, $w_1$ must not be a sink, so we can repeat the above procedure to find $w_2\rightarrow w_3\rightarrow \cdots$, until we find a $w_k$ $(k\geq 1)$ such that $w_k\rightarrow v_2$. Since $G$ has only finite vertices and the above procedure never reaches a sink, so such a $w_k$ must exist, and hence we have $v_1\rightarrow v_2$.

Similarly, we can prove that $\overline{O(v_1)}\subset \overline{O(v_2)}$ and $\overline{O(v_1)}\neq\emptyset$ imply $v_2\rightarrow v_1$.
\end{proof}

\begin{lem}\label{P_2 tilde and P_3}
Let $G$ be a processive graph and $\prec$ a linear order on $E(G)$. If $\prec$ satisfies $(U_1)$ and $(U_2)$, then the following  conditions are equivalent:

$(\widetilde{P_2})$  if $e_1\prec e_2\prec e_3$ and $t(e_1)=s(e_3)$, then either $e_1\rightarrow e_2$ or $e_2\rightarrow e_3$.

$(P_3)$ for any two vertices $v_1$ and $v_2$, $I(v_1)\cap \overline{I(v_2)}\neq\emptyset$ implies that $v_1\rightarrow v_2$ and $O(v_1)\cap \overline{O(v_2)}\neq\emptyset$ implies that $v_2\rightarrow v_1$.

\end{lem}

\begin{proof}
$(\widetilde{P_2})\Longrightarrow (P_3)$.  We only prove the first part of $(P_3)$, the proof of the second part is similar.  First by Lemma \ref{progressive} $(1)$, $I(v_1)\cap \overline{I(v_2)}\neq\emptyset$ implies that $v_2$ is processive.

Now take $e\in I(v_1)\cap \overline{I(v_2)}$,  then $ I(v_2)^-\prec e\prec I(v_2)^+\prec O(v_2)^-$, where the last equality follows from $(U_2)$. Clearly, $t(I(v_2)^-)=v_2=s(O(v_2)^-)$, then by $(\widetilde{P_2})$ we have either $I(v_2)^-\rightarrow e$ or $e\rightarrow O(v_2)^-$.  If $I(v_2)^-\rightarrow e$, then $I(v_2)^+\rightarrow e$, which contradicts $e\prec I(v_2)^+$ and $(U_1)$. So we must have $e\rightarrow O(v_2)^-$, which implies that $v_1\rightarrow v_2$.

$(P_3)\Longrightarrow (\widetilde{P_2})$.  Assume $e_1\prec e_2\prec e_3$ and $v=t(e_1)=s(e_3)$. By $(U_2)$, $\overline{E(v)}=\overline{I(v)}\sqcup \overline{O(v)}$.  Then $e_2\in [e_1,e_3]\subseteq \overline{E(v)}$ implies that  either $e_2\in \overline{I(v)}$ or $e_2\in \overline{O(v)}$. If $e_2\in \overline{I(v)}$, then $e_2\in I(t(e_2))\cap \overline{I(v)}\neq\emptyset$. So by $(P_3)$, we have $t(e_2)\rightarrow v$, hence $e_2\rightarrow e_3$. Similarly, $e_2\in \overline{O(v)}$ implies that $e_1\rightarrow e_2$.
\end{proof}

Now we give several characterizations of POP-graphs.

\begin{thm}\label{PP and AUP equivalence}
Let $G$ be a processive graph with a linear order $\prec$ on $E(G)$ satisfying $(P_1)$. Then the following statements are equivalent:

$(1)$ $(G,\prec)$ is a POP-graph.

$(2)$ $\prec$ satisfies $(U_2)$ and $(P_3)$.

$(3)$ $(G,\prec)$ is an anchored \UPO graph.

$(4)$ $\prec$ satisfies $(U_2)$, $(U_3)$ and $(U_4)$.
\end{thm}
\begin{proof}

$(1)\Longleftrightarrow (2)$.  By Lemma \ref{P_2 tilde and P_3} and the fact that $(P_2)\Longleftrightarrow (\widetilde{P_2})$, we see that $(P_2)\Longleftrightarrow (P_3)$ under $(U_1)$  and $(U_2)$. Since $\prec$ satisfies $(P_1)=(U_1$), then to prove $(1)\Longleftrightarrow (2)$ we only need to prove $(P_1)+(P_2)\Longrightarrow (U_2)$.

In fact, let $v$ be a processive vertex of $G$, $e_1=I(v)^+$ and $e_2=O(v)^-$. Clearly, $e_1\rightarrow e_2$, and by $(P_1)$, $e_1\prec e_2$. Now we prove $e_2=e_1+1$ by contradiction. Suppose there exists an edge $e$ with $e_1\prec e\prec e_2$, then by $(P_2)$ we have either $e_1\rightarrow e$ or $e\rightarrow e_2$. If $e_1\rightarrow e$, then there must exist an edge $e'\in O(v)$ such that $e'\rightarrow e$ or $e'=e$, which follows $e'\preceq e$ by $(P_1)$. Hence $e'\prec e_2$, which contradicts the facts that $e'\in O(v)$ and $e_2=O(v)^-$. Similarly, $e\rightarrow e_2$ also leads a contradiction.

$(1)\Longrightarrow (3)$. We have proved $(P_1)+(P_2)\Longrightarrow (U_2)$ and  Lemma  \ref{property of PPG}  shows that $(P_1)+(P_2)\Longrightarrow (A)$, thus to prove $(1)\Longrightarrow (3)$ we only need to prove $(P_1)+(P_2)\Longrightarrow (U_3)$.  By $(1)\Longleftrightarrow (2)$, it suffices to show  $(P_3)\Longrightarrow (U_3)$.

We prove this by contradiction. Suppose $I(v_1)\cap \overline{I(v_2)}\neq \emptyset$ and $\overline{I(v_1)}\not\subseteq \overline{I(v_2)}$. On one hand, by $(P_3)$, $I(v_1)\cap \overline{I(v_2)}\neq \emptyset$ implies $v_1\rightarrow v_2$.
On the other hand, $I(v_1)\cap \overline{I(v_2)}\neq \emptyset$ implies that $\overline{I(v_1)}\cap \overline{I(v_2)}\neq \emptyset$.
Assume $\overline{I(v_1)}=[e_1, e_2]$ and $\overline{I(v_2)}=[h_1, h_2]$.  Since $\overline{I(v_1)}\not\subseteq \overline{I(v_2)}$, so we have either $e_1\prec h_1\prec e_2$ or $e_1\prec h_2\prec e_2$. Both cases imply $I(v_2)\cap \overline{I(v_1)}\neq \emptyset$, then  by $(P_3)$, we have $v_2\rightarrow v_1$, a contradiction with the acyclicity of $G$.
The second part of $(U_3)$ can be proved similarly.

$(3)\Longrightarrow (2)$.  This is a direct consequence of the fact that $(U_3)+(A)\Longrightarrow (P_3)$.

$(3)\Longleftrightarrow(4)$. This is a direct consequence of Lemma \ref{A and U_4}.

\end{proof}

\section{CPP-extension}
In this section, we introduce the notion of a CPP extension for a directed graph, and show that CPP extensions are naturally in bijective with upward planar orders.
\begin{defn}\label{cpp}
A \emph{canonical processive planar extension},  or \emph{\CPP extension},  of a directed graph $G$ is a POP-graph $(\Gamma, \prec)$ together with an embedding $\phi\colon G\rightarrow \Gamma$, such that

$(E_1)$ $\phi_0(V(G))= V(\Gamma)-(S(\Gamma)\sqcup T(\Gamma))$;

$(E_2)$ $|E(\Gamma)|=|E(G)|+|S(G)|+|T(G)|$, where $|X|$ denotes the cardinality of $X$;

$(E_3)$ $I(\phi_0(v))\cap O(\phi_0(w))= \emptyset$ for any $v\in S(G)$ and $w\in T(G)$;

$(E_4)$  $e\in E(\Gamma)-(\phi_1(E(G))\cup I(\Gamma)\cup O(\Gamma))$ implies that either $O(s(e))^-\prec e \prec O(s(e))^+$ or $I(t(e))^-\prec e \prec I(t(e))^+$.
\end{defn}
Clearly, $G$ must be acyclic if it has a CPP extension. Any \CPP extension of $G$ is obtained from $G$ by adding some new vertices and edges, with local configurations as listed in Fig. $11$. Since $\Gamma$ is processive, $(E_1)$ says that the vertices of $G$  exactly correspond to the processive vertices of $\Gamma$.   $(E_2)$ says that the number of newly added edges are exactly the number $|S(G)|+|T(G)|$ of sources and sinks, and $(E_3)$ says that any newly added edge should not connect a source and a sink of $G$. $(E_4)$ says that if a newly added edge is neither an input nor output edge of $\Gamma$, then its local configuration should be the case $(2)$ or $(4)$ in Fig. $11$.

\begin{rem} In general, \CPP extensions may not exist; and it may not be unique even if it exists. While for a \UPO graph, there exists a unique compatible \CPP extension; see Theorem \ref{P1} below.
\end{rem}

 The following lemma characterizes the newly added edges.

\begin{lem}\label{CPP-1}
Let $G$ be an acyclic directed graph and $\phi\colon G\rightarrow (\Gamma,\prec)$ a \CPP extension. Then
$$\bigsqcup_{v\in S(G)}I(\phi_0(v)) \bigsqcup\bigsqcup_{w\in T(G)}O(\phi_0(w))=E(\Gamma)-\phi_1(E(G)),$$ and  $|I(\phi_0(v))|=1$ for any $v\in S(G)$, and $|O(\phi_0(w))|=1$ for any $w\in T(G)$.
\end{lem}

\begin{proof}
 Assume $S(G)=\{v_1,\cdots,v_m\}$ and $T(G)=\{w_1,\cdots, w_n\}$. By $(E_1)$, $\phi_0(v_k)$ $(1\leq k\leq m)$ and $\phi_0(w_l)$ $(1\leq l\leq n)$ are all processive vertices of $\Gamma$.  Thus $I(\phi_0(v_k))$ $(1\leq k\leq m)$ and $O(\phi_0(w_l))$ $(1\leq l\leq n)$ are not empty. Clearly, $$\bigcup_{1\leq k\leq m}I(\phi_0(v_k)) \bigcup \bigcup_{1\leq l\leq n}O(\phi_0(w_l))\subseteq E(\Gamma)-\phi_1(E(G)),$$ and $I(\phi_0(v_k))\cap I(\phi_0(v_l))= \emptyset$ for any  $1\leq k< l\leq m$ and $O(\phi_0(w_k))\cap O(\phi_0(w_l))= \emptyset$  for any $1\leq k<l\leq n$.  Then  by $(E_2)$ the cardinal number of $E(\Gamma)-\phi_1(E(G))$ is  $m+n$ and by $(E_3)$ the cardinal number of $\bigcup_{1\leq k\leq m}I(\phi_0(v_k)) \bigcup \bigcup_{1\leq l\leq n}O(\phi_0(w_l))$ is also  $m+n$. So we have  $\bigsqcup_{1\leq k\leq m}I(\phi_0(v_k)) \bigsqcup \bigsqcup_{1\leq l\leq n}O(\phi_0(w_l))=E(\Gamma)-\phi_1(E(G))$ and  $|I(\phi_0(v_k))|=1$ $(1\leq k\leq m)$, $|O(\phi_0(w_l))|=1$ $(1\leq l\leq n)$.
\end{proof}

Let $\prec_G$ be the linear order on $E(G)$ induced from $\prec$.
The following lemma is a direct consequence of $(E_4)$, which says that some properties of an edge of $G$ with respect to $\prec_G$ are preserved by the embedding $\phi: G\rightarrow \Gamma$.
\begin{lem}\label{Claim}
Let $G$ be an acyclic directed graph and $\phi \colon G\rightarrow (\Gamma,\prec)$ a \CPP extension of $G$.
For any edge  $e$ of $G$, we set $e'=\phi_1(e)$. Then we have:

$(1)$ $e=I(t(e))^-$ in $(G,\prec_G)$ $\Longleftrightarrow$  $e'=I(t(e'))^-$ in $(\Gamma, \prec)$;

$(2)$ $e=I(t(e))^+$ in $(G,\prec_G)$ $\Longleftrightarrow$  $e'=I(t(e'))^+$ in $(\Gamma, \prec)$;

$(3)$ $e=O(s(e))^-$ in $(G,\prec_G)$ $\Longleftrightarrow$  $e'=O(s(e'))^-$ in $(\Gamma, \prec)$;

$(4)$ $e=O(s(e))^+$ in $(G,\prec_G)$ $\Longleftrightarrow$  $e'=O(s(e'))^+$ in $(\Gamma, \prec)$.

\end{lem}

The following result shows that for a UPO-graph, there is a \CPP extension uniquely determined by its upward planar order.
\begin{thm}\label{P1}
Let $(G,\prec)$ be  a  \UPO graph. Then there exists a unique \CPP extension $\phi: G\rightarrow (\overline{G},\overline{\prec})$ such that for any $e_1, e_2\in E(G)$, $e_1\prec e_2$ implies that $\phi_1(e_1)\overline{\prec} \phi_1(e_2)$.
\end{thm}
\begin{proof}
We construct a POP-graph $(\overline{G},\overline{\prec})$ by adding vertices and edges to $(G,\prec)$ just in the ways of Theorem \ref{L_1_ext}.

$(1)$ For each $v\in S(G)$, if $U(v)=\emptyset$, we add to $G$ a source $v^-$ and an input edge $e_v$ with $s(e_v)=v^-$, $t(e_v)=v$. Otherwise, we add to $G$ an edge $e_v$ with $s(e_v)=U(v)^-$, $t(e_v)=v$. For both cases, we set the order $e_v=O(v)^--1$.

$(2)$ For each $v\in T(G)$, if $D(v)=\emptyset$, we add to $G$ a sink $v^+$ and an output edge $e_v$ with $s(e_v)=v$, $t(e_v)=v^+$. Otherwise, we add to $G$ an edge $e_v$ with $s(e_v)=v$, $t(e_v)=D(v)^-$. For both cases, we set the order $e_v=I(v)^++1$.

Clearly, the order of adding edges are unimportant and  the above construction produces a unique  processive  graph $\overline{G}$, a unique linear order $\overline{\prec}$ on $E(\overline{G})$ and a unique embedding $\phi:G\rightarrow \overline{G}$ which preserves the orders on edges and satisfies $(E_1), (E_2), (E_3), (E_4)$.
Iteratively applying Theorem \ref{L_1_ext}, we see that $(\overline{G},\overline{\prec})$ is a \UPO graph.

To show that $(\overline{G},\overline{\prec})$ is a POP-graph, by Theorem \ref{PP and AUP equivalence}, it suffices to show that $(\overline{G},\overline{\prec})$ satisfies $(U_4)$. We prove this by contradiction. Suppose there exists a  processive vertex $v$ of $\overline{G}$ with $I(\overline{G})\cap \overline{O(v)}\neq\emptyset$. Clearly, $v\in V(G)$. Take an edge $e\in I(\overline{G})\cap \overline{O(v)}$ and set $w=t(e)$. By the construction of $(\overline{G},\overline{\prec})$, $e\in I(\overline{G})$ implies that $w\in S(G)$, $U(w)=\emptyset$ in $(G,\prec)$, $e=e_{w}$, and $O(w)^-=e+1$. Then $e\in \overline{O(v)}$ implies that $O(v)^-\overline{\prec} O(w)^-\overline{\preceq}O(v)^+$. If $O(v)^+= O(w)^-$, then $v=w$, and hence $e\in \overline{O(v)}\cap I(v)$ in $(\overline{G},\overline{\prec})$, which contradicts the fact that $\overline{\prec}$ satisfies $(U_2)$. So we must have $O(v)^-\overline{\prec} O(w)^-\overline{\prec}O(v)^+$, which means $O(w)\cap \overline{O(v)}\neq\emptyset$ in $(\overline{G},\overline{\prec})$. Then by $(U_3)$ for $\overline{\prec}$,  we have $\overline{O(w)}\subset \overline{O(v)}$ in $(\overline{G},\overline{\prec})$, which, by the construction of $(\overline{G},\overline{\prec})$, implies that $\overline{O(w)}\subset \overline{O(v)}$ in $(G,\prec)$, that is, $v\in U(w)\neq\emptyset$ in $(G,\prec)$, a contradiction.
Similarly, we can prove that for any  processive  vertex $v$ of $\overline{G}$, $O(\overline{G})\cap \overline{I(v)}=\emptyset$.

Now we show the uniqueness of the CPP-extension. Suppose $\varphi\colon G\rightarrow (G_1,\prec_1)$ is a \CPP extension of $G$ that preserves the upward planar orders. By Lemma \ref{CPP-1}, for any $e\in E(G_1)-\varphi_1(E(G))$, there exists a unique $v\in S(G)\sqcup T(G)$, such that $\{e\}=I(\varphi_0(v))$ or $\{e\}=O(\varphi_0(v))$. Since $(G_1,\prec_1)$ is a \UPO graph, then $e=O(t(e))^-+1$ or $e=I(s(e))^++1$. Comparing to the construction of $(\overline{G},\overline{\prec})$, it is not difficult to see that there exists a canonical order-preserving isomorphism $\lambda\colon G_1\rightarrow \overline{G}$ such that $\phi=\lambda\circ \varphi$. Thus all order-preserving \CPP extensions of $(G,\prec)$ are canonically isomorphic to each other.
\end{proof}

Conversely, a CPP extension always induces an upward planar order.
\begin{prop}\label{P2}
Any \CPP extension of an acyclic directed graph $G$ induces an upward planar order on $E(G)$.
\end{prop}
\begin{proof}
Let $\phi\colon G\rightarrow (\Gamma,\prec)$ be a \CPP extension of $G$. By Proposition \ref{restriction}, the induced order $\prec_G$ is a linear order satisfying $(U_1)$ and $(U_3)$. To show that $\prec_G$ satisfies $(U_2)$, by $(U_1)$, it suffices to show that for any  processive  vertex $v$ of $G$, $I(v)^++1=O(v)^-$, which follows from Lemma \ref{Claim}.
\end{proof}

As a direct consequence of Theorem \ref{P1} and Proposition \ref{P2}, the following is our main result in this section.
\begin{thm}\label{T2}
For any acyclic directed graph $G$, there is a bijection between the set of upward planar orders and the set of \CPP extensions.
\end{thm}

\section{Justifying UPO-graph}

In this section, we will prove our main result, Theorem \ref{T3}, which shows that upward planar orders indeed characterizes upward planarity.

\begin{thm}\label{T3} Any \UPO graph has a unique upward planar drawing up to planar isotopy; and conversely, there is an upward planar order on $E(G)$ for any upward plane graph $G$.
\end{thm}

The first part follows from Theorem \ref{T2} and Theorem \ref{T1}.  For the converse part, we need the following proposition, which is a geometric counterpart of Theorem \ref{P1}.
\begin{prop}\label{P3}
Let $G$ be an upward plane graph. Then there exists a PPG $\Gamma$ and a (geometric) embedding $\phi\colon G\rightarrow \Gamma$, such that

$(1)$ $\phi_0(V(G))=V(\Gamma)-(S(\Gamma)\sqcup T(\Gamma))$;

$(2)$ $|E(\Gamma)|=|E(G)|+|S(G)|+|T(G)|$;

$(3)$ for any $v\in S(G)$ and $w\in T(G)$,  $I(\phi_0(v))\cap O(\phi_0(w))= \emptyset$;

$(4)$  $e\in E(\Gamma)-(\phi_1(E(G))\cup I(\Gamma)\cup O(\Gamma))$ implies that either $O(s(e))^-<e < O(s(e))^+$ or $I(t(e))^-< e < I(t(e))^+$, where the linear orders are given by the polarization  of $\Gamma$.
\end{prop}

\begin{proof}
We want to extend $G$ into a PPG. Let $v_1,\cdots, v_n$ be an ordered list of $V(G)$,  with $(X_1,Y_1),\cdots,(X_n,Y_n)$ as their coordinates, such that $Y_1\geq\cdots \geq Y_n$. Then $G$ is contained in the box $D=[K-1,L+1]\times [Y_1+1,Y_n-1]$, where $K=min\{X_1,\cdots,X_n\},\ L=max\{X_1,\cdots,X_n\}$.

Assume $T(G)=\{ v_{\alpha_1},\cdots, v_{\alpha_\mu}\}$ with $1\leq\alpha_1<\cdots <\alpha_\mu\le n$. Clearly $v_n=v_{\alpha_\mu}$ is a sink. We will inductively eliminate all the sinks of $G$ by adding suitable new edges and vertices.

First, we add a vertex $v_n^+$ and an edge $h=[v_n,v_n^+]$ to $G$, where the coordinate of $v_n^+$ is $(X_n, Y_n-1)$ and  $h=[v_n,v_n^+]$ is the segment with  $s(h)=v_n$, $t(h)=v_n^+$. Denote the resulting upward plane graph as $G_1$.

Then we move to $v_{\alpha_{\mu-1}}$. If $Y_{\alpha_{\mu-1}}=Y_{\alpha_\mu}$, just as above, we add a vertex $v_{\alpha_{\mu-1}}^+$ and an edge $h_1=[v_{\alpha_{\mu-1}},v_{\alpha_{\mu-1}}^+]$ to $G_1$. Otherwise,  $Y_{\alpha_{\mu-1}}>Y_{\alpha_\mu}$. Then we consider the horizontal line $y= Y_{\alpha_{\mu-1}}$ and the set of its intersection points with $G_1$. There are three cases: $(1)$ there is an intersection point on the left of $v_{\alpha_{\mu-1}}$; (2) there is an intersection point on the right of $v_{\alpha_{\mu-1}}$; and (3) $v_{\alpha_{\mu-1}}$ is the unique intersection point of the line with $G_1$.

Case (1):  We consider the strip of the plane delimited by horizontal lines $y= Y_{\alpha_{\mu-1}}$ and $y= Y_{\alpha_{\mu-1}}-\varepsilon$, where $\varepsilon>0$ is small enough so that the strip contains no vertices in its interior, and the strip is divided by the edges of $G_1$ into (at least two) connected regions bounded by vertically monotonic curves.

Let $e_0$ be the edge of $G_1$ such that it is on the boundary of the region that contains $v_{\alpha_{\mu-1}}$ and on the left of $v_{\alpha_{\mu-1}}$ (the assumption in (1) guarantees the existence of $e_0$).
Then  either $e_0=h$, or there exists a unique directed path $e_0e_1e_2\cdots e_r$, such that $e_i=O(s(e_i))^+$ for all $1\le i\le r$ and $t(e_r)=v_n^+$ (equivalently, $e_r=h$), where the existence and the uniqueness of such path are guaranteed by the ordering of $T(G)$ and the requirement $e_i=O(s(e_i))^+$, respectively. Then we have two subcases.

Subcase $(1.1)$: For all $0\le i\le r$, $e_i=I(t(e_i))^+$ with respect to the polarization of $G$.

In this case, we add to $G_1$ a vertex $v_{\alpha_{\mu-1}}^+$  with vertical ordinate $Y_n-1$, on the right of and close enough to $v_n^+$; add an edge $h_1$ with $s(h_1)=v_{\alpha_{\mu-1}}$,  $t(h_1)=v_{\alpha_{\mu-1}}^+$, and draw it into $G_1$ as a monotonic curve on the right of and close enough to the directed path $e_0e_1\cdots e_r$; see Fig. $12$.

\begin{center}
\begin{center}
\begin{tikzpicture}[scale=0.7]

\node (v1) at (0,0) {};
\node (v2) at (6,0) {};
\node (v3) at (0,-1.5) {};
\node (v4) at (6,-1.5) {};
\node (v5) at (0,-7.5) {};
\node (v6) at (6,-7.5) {};
\draw  (v1) edge (v2);
\draw  (v3) edge (v4);
\draw  (v5) edge (v6);
\node (v7) at (1.5,-6) {};
\node [scale=0.6][left] at(1.5,-6) {$v_n$};
\draw[fill] (1.5,-6) circle [radius=0.07];
\node (v8) at (1.5,-7.5) {};
\node [scale=0.6][below] at(1.2,-7.5) {$v_n^+$};
\draw[fill] (1.5,-7.5) circle [radius=0.07];
\draw  (1.5,-6) -- (1.5,-7.5)[postaction={decorate, decoration={markings,mark=at position .5 with {\arrow[black]{stealth}}}}];
\node [scale=0.6][above]at (4.5,0) {$v_{\alpha_{\mu-1}}$};
\draw[fill] (4.5,0) circle [radius=0.07];
\node (v9) at (3.5,-2) {};
\node[scale=0.6] [left] at (3,-2.5) {$e_1$};
\draw[fill] (3.5,-2) circle [radius=0.07];

\node (v10) at (3,-3.5) {};
\node[scale=0.6] [left] at(2.3,-3.7) {$e_2$};
\draw[fill] (3,-3.5) circle [radius=0.07];

\draw  plot[smooth, tension=.7] coordinates {(2.5,1) (3.5,0.5) (3.5,0) (3,-0.5) (3,-1) (3,-1.5) (3.5,-2)}[postaction={decorate, decoration={markings,mark=at position .5 with {\arrow[black]{stealth}}}}];
\draw  plot[smooth, tension=.7] coordinates {(3.5,-2) (3.4,-2.6) (3.4,-3) (3,-3.5)}[postaction={decorate, decoration={markings,mark=at position .5 with {\arrow[black]{stealth}}}}];
\draw  plot[smooth, tension=.7] coordinates {(3,-3.5) (2.4,-4) (2.2,-4.2) (2,-4.6) (1.8,-5.2) (1.6,-5.8) (v7)}[postaction={decorate, decoration={markings,mark=at position .5 with {\arrow[black]{stealth}}}}];
\node [scale=0.6]at (7,0) {$Y_{\alpha_{\mu-1}}$};
\node[scale=0.6] at (7.4,-1.5) {$Y_{\alpha_{\mu-1}}-\varepsilon$};
\node [scale=0.6]at (7,-7.5) {$Y_n-1$};
\draw  plot[smooth, tension=.7] coordinates {(4.5,0)(3.4,-0.6) (3.4,-1.2) (3.8,-2.4) (3.4,-4) (2.6,-4.6) (2.2,-5.6) (2,-6.8) (2,-7.5)}[postaction={decorate, decoration={markings,mark=at position .5 with {\arrow[black]{stealth}}}}];
\draw[fill] (2,-7.5) circle [radius=0.07];

\node [below] at(2.6,-7.5) {$v_{\alpha_{\mu-1}}^+$};
\node [scale=0.6]at (2.5,-0.4) {$e_0$};
\node[scale=0.6][rotate=-20] at (1.5,-4.6) {$\vdots$};
\node [scale=0.6]at (0.5,-6.8) {$e_r=h$};

\node (v11) at (3,-1.8) {};
\node (v12) at (3,-2.2) {};
\node (v14) at (2.4,-3) {};
\node (v13) at (2.8,-3) {};
\node (v15) at (2.4,-3.8) {};
\draw (3,-1.8)-- (3.5,-2)[postaction={decorate, decoration={markings,mark=at position .5 with {\arrow[black]{stealth}}}}];
\draw  (3.5,-2) -- (3,-2.5)[postaction={decorate, decoration={markings,mark=at position .5 with {\arrow[black]{stealth}}}}];
\draw  (2.8,-3)-- (3,-3.5)[postaction={decorate, decoration={markings,mark=at position .5 with {\arrow[black]{stealth}}}}];
\draw  (2.4,-3)--(3,-3.5)[postaction={decorate, decoration={markings,mark=at position .5 with {\arrow[black]{stealth}}}}];
\draw(3,-3.5) -- (2.4,-3.8)[postaction={decorate, decoration={markings,mark=at position .5 with {\arrow[black]{stealth}}}}];

\node (v16) at (1.4,-5.4) {};
\node (v17) at (1,-5.4) {};
\draw   (1.4,-5.4) --(1.5,-6)[postaction={decorate, decoration={markings,mark=at position .5 with {\arrow[black]{stealth}}}}];
\draw   (1,-5.4)  -- (1.5,-6)[postaction={decorate, decoration={markings,mark=at position .5 with {\arrow[black]{stealth}}}}];

\node [scale=0.6]at (4,-3.6) {$h_1$};
\end{tikzpicture}

\end{center}
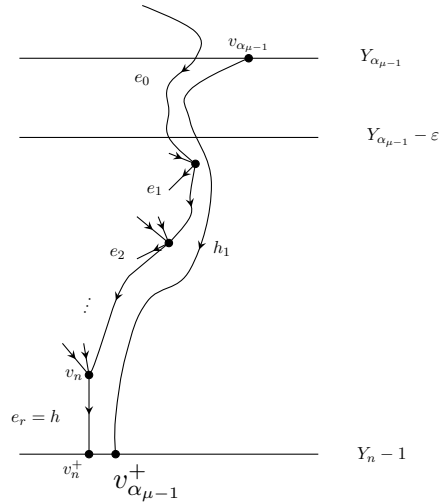
Figure  $12$. Drawing of $h_1$ in subcase $(1.1)$.
\end{center}

Subcase $(1.2)$: There exists some $i\in [0,\cdots, r-1]$ such that $e_i<I(t(e_i))^+$, and $e_j=I(t(e_j))^+$ for all $j<i$.

In this case, we  add to $G_1$ an edge $h_1$ with $s(h_1)=v_{\alpha_{\mu-1}}$, $t(h_1)=t(e_i)$, and draw it into $G_1$ as a monotonic curve on the right of and close enough to $e_0e_2\cdots e_i$; see Fig. $13$. Clearly, $e_i < h_1<I(t(e_i))^+$ with respect to the polarization of the resulting upward plane graph.

\begin{center}
\begin{tikzpicture}[scale=0.7]

\node (v1) at (-2.5,1) {};
\node (v2) at (5,1) {};
\draw  (v1) edge (v2);
\node (v19) at (2.5,1) {};
\node [scale=0.6][above]at (2.5,1) {$v_{\alpha_{\mu-1}}$};
\draw[fill] (2.5,1) circle [radius=0.07];
\node (v3) at (-2.5,-1) {};
\node (v4) at (5,-1) {};
\node [scale=0.6]at (6,1) {$Y_{\alpha_{\mu-1}}$};
\node[scale=0.6] at (6,-1) {$Y_{\alpha_{\mu-1}}-\varepsilon$};
\draw  (v3) edge (v4);
\node (v5) at (0.5,-0.5) {};
\draw[fill] (0.5,-0.5) circle [radius=0.07];
\node (v6) at (0,-2.5) {};
\draw[fill] (0,-2.5) circle [radius=0.07];
\node (v7) at (1.5,-5) {};
\draw[fill] (1.5,-5) circle [radius=0.07];

\draw  plot[smooth, tension=.7] coordinates {(1,1.5) (0.5,0.8) (0.4,0.2) (0.5,-0.5) (0.6,-1) (0.6,-1.4) (0.4,-1.9) (v6) (-0.1,-3) (0.1,-3.5) (0.4,-3.8) (0.7,-4) (1,-4.3) (1.2,-4.6) (v7) (1.8,-6.2)}[postaction={decorate, decoration={markings,mark=at position .15 with {\arrow[black]{stealth}}}}][postaction={decorate, decoration={markings,mark=at position .4 with {\arrow[black]{stealth}}}}][postaction={decorate, decoration={markings,mark=at position .7 with {\arrow[black]{stealth}}}}][postaction={decorate, decoration={markings,mark=at position .97 with {\arrow[black]{stealth}}}}];
\node (v8) at (0,-0.1) {};
\node (v9) at (-0.3,-0.2) {};
\draw  (v8) -- (0.5,-0.5)[postaction={decorate, decoration={markings,mark=at position .5 with {\arrow[black]{stealth}}}}];
\draw  (v9) -- (0.5,-0.5)[postaction={decorate, decoration={markings,mark=at position .5 with {\arrow[black]{stealth}}}}];
\node (v10) at (-0.1,-2) {};
\node (v11) at (-0.6,-2.2) {};
\draw  (v10) -- (0,-2.5)[postaction={decorate, decoration={markings,mark=at position .5 with {\arrow[black]{stealth}}}}];
\draw  (v11) -- (0,-2.5)[postaction={decorate, decoration={markings,mark=at position .5 with {\arrow[black]{stealth}}}}];
\node (v12) at (1.8,-4.4) {};
\node (v13) at (2.4,-4.5) {};
\draw  (v12) -- (1.5,-5)[postaction={decorate, decoration={markings,mark=at position .4 with {\arrow[black]{stealth}}}}];
\draw  (v13) -- (1.5,-5)[postaction={decorate, decoration={markings,mark=at position .5 with {\arrow[black]{stealth}}}}];
\node (v14) at (0.1,-0.9) {};
\node (v16) at (-0.7,-2.9) {};
\node (v15) at (-0.4,-3) {};
\node (v18) at (0.8,-5.6) {};
\node (v17) at (1.3,-5.6) {};
\node at (1.8,-5.6) {};
\draw  (0.5,-0.5) -- (v14)[postaction={decorate, decoration={markings,mark=at position .7 with {\arrow[black]{stealth}}}}];
\draw  (0,-2.5)-- (v15)[postaction={decorate, decoration={markings,mark=at position .7 with {\arrow[black]{stealth}}}}];
\draw  (0,-2.5)-- (v16)[postaction={decorate, decoration={markings,mark=at position .8 with {\arrow[black]{stealth}}}}];
\draw (1.5,-5) -- (v17)[postaction={decorate, decoration={markings,mark=at position .7 with {\arrow[black]{stealth}}}}];
\draw  (1.5,-5) -- (v18)[postaction={decorate, decoration={markings,mark=at position .7 with {\arrow[black]{stealth}}}}];
\draw  plot[smooth, tension=.7] coordinates {(v19) (1.6,0.7) (0.7,0) (0.7,-0.7) (0.8,-1.6) (0.4,-2.5) (0.2,-3.2) (0.7,-3.7) (1.3,-4.2) (1.5,-4.6) (v7)}[postaction={decorate, decoration={markings,mark=at position .5 with {\arrow[black]{stealth}}}}];
\node [scale=0.6]at (0,0.3) {$e_0$};
\node [scale=0.6]at (0.2,-1.4) {$e_1$};
\node [scale=0.6]at (-0.4,-3.3) {$e_2$};
\node [rotate=50] at (0.2,-4) {$\vdots$};
\node [scale=0.6]at (1,-4.8) {$e_i$};
\node [scale=0.6]at (2.1,-5.5) {$e_{i+1}$};
\node [scale=0.6]at (0.9,-2.1) {$h_1$};
\end{tikzpicture}

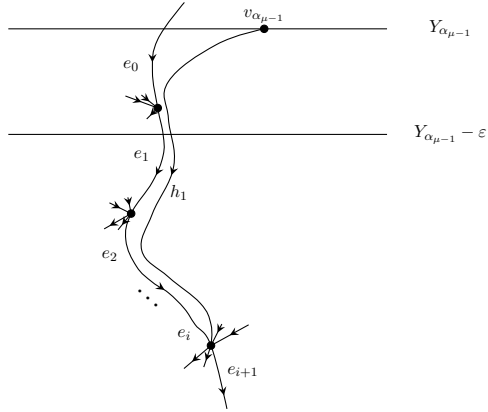
Figure $13$. Drawing of $h_1$ in case $(1.2)$.
\end{center}

Case (2):  This case is similar to Case (1), and we omit it here.

Case (3):  In this case, since $Y_{\alpha_{\mu-1}}>Y_{\alpha_\mu}$, so there must exist a source $v_k\in V(G)$ with $\alpha_{\mu-1}<k<n$ such that there is no edge of $G$ intersecting with the interior of the horizontal strip delimited by horizontal lines $y=Y_{\alpha_{\mu-1}}$ and $y=Y_k$. Take $e_0=[v_{\alpha_{\mu-1}},v_k]$ the segment with $s(e_0)=v_{\alpha_{\mu-1}}$ and $t(e_0)=v_k$. Then we reduce this case to the above case $(1)$ or case $(2)$ of the auxiliary upward plane graph $G_1'=G_1+ \{e_0\}$. Figure $14$ shows one possible subcase of case $(3)$, where $e_0$ is just an auxiliary edge for the drawing of $h_1$ and is not a really added edge of $G_1$.

\begin{center}
\begin{tikzpicture}[scale=0.7]

\node (v1) at (-2.5,1) {};
\node (v2) at (5,1) {};
\draw  (v1) edge (v2);
\node (v19) at (2.5,1) {};
\node [scale=0.6][above]at (1.6,1) {$v_{\alpha_{\mu-1}}$};
\draw[fill] (2.5,1) node (v6) {} circle [radius=0.07];
\node (v3) at (-2.5,-1) {};
\node (v4) at (5,-1) {};
\node [scale=0.6]at (6,1) {$Y_{\alpha_{\mu-1}}$};
\node [scale=0.6]at (6,-1) {$Y_k$};
\draw  (v3) edge (v4);

\node (v5) at (1.8,1.7) {};
\node (v7) at (2.4,1.8) {};
\node (v8) at (3.4,1.6) {};
\draw (v5) --(2.5,1)[scale=0.6,postaction={decorate, decoration={markings,mark=at position .4 with {\arrow[black]{stealth}}}}];
\draw  (v7) -- (2.5,1)[scale=0.6,postaction={decorate, decoration={markings,mark=at position .4 with {\arrow[black]{stealth}}}}];
\draw  (v8) -- (2.5,1)[scale=0.6,postaction={decorate, decoration={markings,mark=at position .4 with {\arrow[black]{stealth}}}}];
\node [scale=0.6] at (2.9,1.7) {$\cdots$};
\node at (0.1,-1) {};
\draw[fill](0.1,-1) node (v9) {}  circle [radius=0.07];

\draw  [dashed](v6)-- (v9)[scale=0.6,postaction={decorate, decoration={markings,mark=at position .45 with {\arrow[black]{stealth}}}}];
\node (v10) at (-2.5,0) {};
\node (v11) at (5,0) {};
\draw  (v10) edge (v11);
\node [scale=0.6]at (6,0) {$Y_{\alpha_{\mu-1}}-\varepsilon$};
\node [scale=0.6]at (-0.3,-0.7) {$v_k$};

\node (v12) at (-2.5,-4.5) {};
\node (v13) at (5,-4.5) {};
\draw  (v12) edge (v13);
\node[scale=0.6] at (6,-4.5) {$Y_n-1$};
\node[scale=0.6] at (1,0.2) {$e_0$};
\node at (0.8,-4.5) {};
\draw[fill] (0.8,-4.5) node (v17) {}circle [radius=0.07];
\node (v14) at (1,-1.8) {};
\draw[fill] (1,-1.8) circle [radius=0.07];
\node (v15) at (1.7,-2.8) {};
\draw[fill] (1.7,-2.8) circle [radius=0.07];
\node (v16) at (0.8,-4) {};
\draw[fill] (0.8,-4) circle [radius=0.07];
\draw  plot[smooth, tension=.7] coordinates {(0.1,-1)(0.5,-1.3) (1,-1.8) (v15) (v16)}[scale=0.6,postaction={decorate, decoration={markings,mark=at position .2 with {\arrow[black]{stealth}}}}][scale=0.6,postaction={decorate, decoration={markings,mark=at position .5 with {\arrow[black]{stealth}}}}][postaction={decorate, decoration={markings,mark=at position .8 with {\arrow[black]{stealth}}}}];
\draw  (0.8,-4) --((0.8,-4.5)[scale=0.6,postaction={decorate, decoration={markings,mark=at position .55 with {\arrow[black]{stealth}}}}];

\node [scale=0.6]at (0.3,-1.5) {$e_1$};
\node [scale=0.6]at (1.1,-2.3) {$e_2$};
\node [scale=0.6][rotate=50]at (1.2,-3.2) {$\cdots$};
\node[scale=0.6] at (0.3,-3.7) {$v_n$};
\node [scale=0.6]at (0.8,-4.9) {$v_n^+$};
\node[scale=0.6] at (0,-4.2) {$e_r=h$};
\node (v18) at (0.5,-1.7) {};
\node (v20) at (0.8,-2.2) {};
\draw  (v18)-- (1,-1.8)[scale=0.6,postaction={decorate, decoration={markings,mark=at position .5 with {\arrow[black]{stealth}}}}];
\draw (1,-1.8)-- (0.8,-2.2)[scale=0.6,postaction={decorate, decoration={markings,mark=at position .7 with {\arrow[black]{stealth}}}}];
\draw  plot[smooth, tension=.7] coordinates {(v6) (1.7,0.1) (0.5,-0.9) (1.1,-1.5) (1.8,-2.4) (2,-3.2) (1.5,-3.9) (1.3,-4.5)}[scale=0.6,postaction={decorate, decoration={markings,mark=at position .7 with {\arrow[black]{stealth}}}}];
\node at (1.3,-4.5) {};
\draw[fill]  (1.3,-4.5) circle [radius=0.07];
\node[scale=0.6] at (1.5,-4.9) {$v_{\alpha_{\mu-1}}^+$};
\node[scale=0.6] at (2.2,-3.3) {$h_1$};
\end{tikzpicture}

Figure $14$. One possible subcase of case $(3)$.
\end{center}

Repeating the above procedure successively for $v_{\alpha_{\mu-2}},v_{\alpha_{\mu-3}},\cdots, v_{\alpha_1}$, we can eliminate all the sinks. Similarly, we can eliminate all the sources. As a result, we get a PPG $\Gamma$ boxed in $D$ and  with $G$ as a subgraph. By the construction the resulting embedding $\psi\colon G\rightarrow \Gamma$ satisfies all the required conditions listed in the proposition. The proof is completed.
\end{proof}

By Theorem \ref{T1}, the (geometric) embedding $\phi\colon G\rightarrow \Gamma$ in Proposition \ref{P3} induces a \CPP extension of $G$, which, by Proposition \ref{P2}, implies the converse part of Theorem \ref{T3}.
\begin{rem}
Note that the extension of $G$ in Proposition \ref{P3} is not unique, so the upward planar order on $E(G)$ is not necessarily unique.
\end{rem}

\begin{rem}\label{Re2}
Note that if $(G,\prec)$ is an anchored UPO-graph, then for any $v\in S(G)$, $U(v)=\emptyset$, and for any $v\in T(G)$, $D(v)=\emptyset$. By the construction in Theorem \ref{P1}, the CPP extension $\phi: G\rightarrow (\overline{G},\overline{\prec})$ satisfies $E(\overline{G})-E(G)=I(\overline{G})\sqcup O(\overline{G})$. Therefore, in the geometric realization of $(G,\prec)$, all sources and sinks of $G$ are drawn on the boundary of the external face.
\end{rem}

Using a fundamental result independently due to F\'{a}ry \cite{[Fa48]} and Wagner \cite{[Wa36]}, we may obtain a combinatorial characterization of (non-directed) planar graphs.

\begin{cor} A (non-directed) graph $G$ has a planar drawing if and only if there exists an orientation on $G$ and an upward planar order on $E(G)$ with respect to the orientation.
\end{cor}

\begin{proof} We need only to show the ``only if'' part, and the other direction is obvious.

We first claim that any simple planar graph $\Gamma$ has an upward drawing. By F\'ary-Wagner theorem, there exists a planar drawing of $\Gamma$ such that all edges are straight line segment. We may rotate the plane an appropriate angle, so that any horizontal line contains at most one vertex of $\Gamma$. This can be done since $\Gamma$ has only finitely many vertices, and hence only finitely many straight lines will contain more than one vertices. Then there exists a (unique) orientation of $\Gamma$ making the resulting drawing an upward planar drawing.

Now the proof follows from the easy fact that a graph
$G$ has a planar drawing if and only if the associated simple graph of $G$ has a planar drawing.
\end{proof}

\section*{Acknowledgement}
This work is partly supported by the National Scientific Foundation of China No.11431010 and 11571329 and  "the Fundamental Research Funds for the Central Universities".

\par


\textbf{Xuexing Lu}\hfill \\  Email: xxlu@mail.ustc.edu.cn

\textbf{Yu Ye} \hfill \\ Email: yeyu@ustc.edu.cn

\end{document}